\pdfoutput=1 %so that arXiv doesn't fudge it
\documentclass[11pt]{amsart}
\usepackage[utf8]{inputenc}
\usepackage{amssymb,amsmath,enumerate,amsthm,enumitem,colonequals,mlmodern,tikz-cd,microtype}
\usepackage[cal=euler,bfcal,bb=px,bfbb]{mathalpha}
\usepackage[top=3.75cm, bottom=3cm, left=3.5cm, right=3.5cm]{geometry}

\usepackage{xcolor}
\colorlet{darkblue}{blue!55!black}
\colorlet{darkcyan}{cyan!50!black}
\colorlet{darkgreen}{green!60!black}

\usepackage{hyperref}
\hypersetup{
    colorlinks=true,
    linkcolor= darkblue,
    urlcolor= darkcyan,
    citecolor= darkgreen,
}

\def\eqref#1{\textcolor{darkblue}{(\ref{#1})}}

\PassOptionsToPackage{hyphens}{url}
\usepackage{hyperref} 
\hypersetup{bookmarksdepth=2}

\usepackage[nameinlink]{cleveref} %the nameinlink option can be left out; see what suits
\Crefformat{section}{#2\S#1#3} % to get \S instead of \Section
\Crefmultiformat{section}{#2\S\S#1#3}{ and~#2#1#3}{, #2#1#3}{, and~#2#1#3}

\newcounter{intro}

\newtheorem{introthm}[intro]{Theorem}

\theoremstyle{plain}
\newtheorem{theorem}{Theorem}[section]
\newtheorem{lemma}[theorem]{Lemma}
\newtheorem{corollary}[theorem]{Corollary}
\newtheorem{proposition}[theorem]{Proposition}

\theoremstyle{definition}

\newtheorem{construction}[theorem]{Construction}

\newtheorem{definition}[theorem]{Definition}
\newtheorem{example}[theorem]{Example}

\newtheorem{question}[theorem]{Question}
\newtheorem{remark}[theorem]{Remark}

\newtheorem*{ack}{Acknowledgements}

\numberwithin{equation}{section}
\numberwithin{theorem}{section}

\title[Descent, generation, coherent algebras]{Descent and generation for \\ noncommutative coherent algebras over schemes}

\author[T.~De Deyn]{Timothy De Deyn}
\address{T.~De Deyn,
School of Mathematics and Statistics,
University of Glasgow, 
Glasgow G12 8QQ,
United Kingdom}
\email{timothy.dedeyn@glasgow.ac.uk}

\author[P.~Lank]{Pat Lank}
\address{P.~Lank,
Dipartimento di Matematica “F. Enriques”, Universit\`{a} degli Studi di Milano, Via Cesare
Saldini 50, 20133 Milano, Italy}
\email{plankmathematics@gmail.com}

\author[K. ~Manali Rahul]{Kabeer Manali Rahul}
\address{K. ~Manali Rahul,
Center for Mathematics and its Applications, 
Mathematical Science Institute, Building 145, 
The Australian National University, 
Canberra, ACT 2601, Australia}
\email{kabeer.manalirahul@anu.edu.au}

\date{\today}



\begin{document}
    
\begin{abstract}
    Our work shows forms of descent, in the fppf, h and \'{e}tale topologies, for strong generation of the bounded derived category of a noncommutative coherent algebra over a scheme.
    Even for (commutative) schemes this yields new perspectives. 
    As a consequence we exhibit new examples where these bounded derived categories admit strong generators. 
    We achieve our main results by leveraging the action of the scheme on the coherent algebra, allowing us to lift statements into the noncommutative setting.
    In particular, this leads to interesting applications regarding generation for Azumaya algebras.
\end{abstract}

\maketitle
\tableofcontents

%%%%%%%%%%%%%%%%%%%%%%%%%%%%%%%%
\section{Introduction}
\label{sec:intro}
%%%%%%%%%%%%%%%%%%%%%%%%%%%%%%%%

Triangulated categories (and their enhancements) are ubiquitous in numerous branches of modern mathematics.
One of the first natural things to study is how these categories can be generated, under the naturally allowed operations, from a finite amount of data.
For example there are interesting algebro-geometric consequences to these types of `generation questions'.
These include, amongst others, closedness of the singular locus of a scheme \cite{Iyengar/Takahashi:2019, Dey/Lank:2024a}, characterizations of various singularities arising in birational geometry \cite{Lank/Venkatesh:2024,Lank/McDonald/Venkatesh:2025}, and a characterization of regularity for `nice enough' scheme \cite{Bondal/VandenBergh:2003,Neeman:2021b}.
Additionally, strong generation allows for representability statements of cohomological functors \cite{Bondal/VandenBergh:2003, Rouquier:2008,Neeman:2018b}.

In this work we investigate the \emph{strong} generation of triangulated categories arising naturally in (noncommutative) algebraic geometry.
There are a various approaches to `noncommutative algebraic geometry', and we refrain from spelling out any reasonable survey of the literature due to its breadth. 
However, the overarching theme is to use noncommutative structures to extract information from geometric objects or, the other way round, studying noncommutative objects using geometric means.
This area of mathematics has enjoyed a wide range of connections to other fields, including, for example, mathematical physics, symplectic geometry and representation theory. 

The approach adopted in the current work is the study of `mildly' noncommutative schemes: these are (usual, commutative) schemes equipped with noncommutative structure sheaves. 
That is, a \textit{(mild) noncommutative scheme} consists of a pair $(X,\mathcal{A})$ where $X$ is a scheme and $\mathcal{A}$ is a quasi-coherent algebra over $X$. 
We refer to \Cref{sec:nc_scheme} for some background.
Of particular interest, when $\mathcal{A}$ is coherent, is the associated \textit{bounded derived category} $D^b_{\operatorname{coh}}(\mathcal{A})$, i.e.\ the derived category of complexes with bounded and coherent cohomology.
We study its strong generation and how this behaves with respect to various topologies.

Before delving into our results, let us briefly recall some terminology. 
Consider a collection of objects $\mathcal{S}$ in a triangulated category $\mathcal{T}$. The smallest triangulated subcategory of $\mathcal{T}$ containing $\mathcal{S}$ and  closed under direct summands is denoted $\langle \mathcal{S} \rangle$ and is called the \textit{thick closure} of $\mathcal{S}$. Any object in $\langle \mathcal{S} \rangle$ can be \textit{finitely built} from $\mathcal{S}$ using only a finite number of shifts, cones and direct summands.
When we want to restrict the number of cones used, we write $\langle \mathcal{S} \rangle_{n+1}$ for the subcategory of objects in $\mathcal{T}$ which can be finitely built from $\mathcal{S}$ using finite coproducs, direct summands, shifts, and \textit{at most} $n$ cones. An object $G$ in $\mathcal{T}$ is called a \textit{strong generator} if $\langle G \rangle_{n+1} = \mathcal{T}$ for some $n\geq 0$, and the smallest such $n$ is called the \textit{Rouquier dimension} of $\mathcal{T}$.
See \Cref{sec:generation} for more details.

%%%%%%%%%%%%%%%%%%%%%%%%%%%%%%%%
\subsection{Descent for Grothendieck topologies}
\label{sec:intro_descent}
%%%%%%%%%%%%%%%%%%%%%%%%%%%%%%%%
Roughly descent describes how to obtain global statements by `glueing' local statements. 
The primordial form of descent in algebraic geometry goes by lifting properties from an (affine) open cover, i.e.\ the Zariski topology. 
In fact, this is how a lot of properties of schemes are defined.

When it comes to generation this has been very fruitful for finding compact generators by reducing Zariski locally \cite{Neeman:1996,Bondal/VandenBergh:2003,Rouquier:2008} and more recently fppf locally \cite{Toen:2012,Hall/Rydh:2017}.
Our first result below shows Zariski descent for the existence of strong generators in bounded derived categories and is a noncommutative generalization of \cite[Theorem D]{Lank:2024}.
It yields new examples of strong generators existing in the noncommutative setting (see e.g.\ \Cref{ex:els_global}).

\begin{introthm}[see \Cref{thm:nc_local_to_global_affine_cover}]\label{introthm:nc_local_to_global_affine_cover}
    Let $X$ be a separated Noetherian scheme and $\mathcal{A}$ a coherent $\mathcal{O}_X$-algebra. 
    The following are equivalent
    \begin{enumerate}
        \item $D^b_{\operatorname{coh}}(\mathcal{A})$ admits a strong generator,
        \item $D^b_{\operatorname{coh}}(\mathcal{A}|_U)$ admits strong generator for each affine open $U\subseteq X$.
    \end{enumerate}
\end{introthm}

Life being life, it is often more convenient, and natural, to allow for arbitrary direct sums when constructing generators.
This was a key insight in \cite{Aoki:2021, Neeman:2021b}. 
Recall that an object $G$ in a triangulated category $\mathcal{T}$ admitting small coproducts is said to be a \textit{strong $\oplus$-generator} if any objects in $\mathcal{T}$ can be obtained from $G$ using arbitrary small coproducts, direct summands, shifts, and \textit{at most} $n$ cones, see also \Cref{sec:generation}. 
The following shows that one can check the existence of strong $\oplus$-generators locally in various topologies (see \Cref{sec:h-top} for a brief recollection of the h topology for Noetherian schemes, and \cite[\href{https://stacks.math.columbia.edu/tag/021L}{Tag 021L}]{stacks-project} for the fppf topology).

\begin{introthm}[see \Cref{thm:nc_local_to_global_h_cover_big}]\label{introthm:nc_local_to_global_h_cover_big}
    Let $X$ be a separated Noetherian scheme and $\mathcal{A}$ a coherent $\mathcal{O}_X$-algebra. The following are equivalent
    \begin{enumerate}
        \item $D_{\operatorname{qc}}(\mathcal{A})$ admits a strong $\oplus$-generator with bounded and coherent cohomology,
        \item\label{item2} there exists an fppf covering $\{ f_i \colon X_i \to X \}_{i}$ with, for each $i$, $D_{\operatorname{qc}}(f_i^\ast \mathcal{A})$ admitting a strong $\oplus$-generator with bounded and coherent cohomology.
    \end{enumerate}
    Additionally, if $\mathcal{A}$ is flat over $X$, then the above are also equivalent to
    \begin{enumerate}[resume]
        \item\label the same as \eqref{item:hcover2} but for $\{ f_i \colon X_i \to X \}_{i}$ a (finite) h covering.
    \end{enumerate}
\end{introthm}

Generally speaking, within the current literature, constructing strong $\oplus$-generators for the `big' category that are bounded and coherent, or perfect, is `easier' than constructing strong generators for the respective `small' categories (the former always gives the latter).
Yet, in general one definitely cannot expect the converse---that strong generators for the small category yield strong $\oplus$-generators for the big category; we elaborate a bit on their relation in \Cref{sec:big_vs_small_elephant} at the end.
Regardless, this is the reason for considering results concerning the descent of strong $\oplus$-generators. 

Let us quickly mention the two key technical ingredients for the proof of \Cref{introthm:nc_local_to_global_h_cover_big}:
\begin{itemize}
    \item For any h cover $f\colon Y\to X$, the derived pushforward of the structure sheaf $\mathbb{R}f_\ast \mathcal{O}_Y$ is a `descendable object', in the sense \cite[Definition 3.18]{Mathew:2016}, in $D_{\operatorname{qc}}(X)$ (\cite[Definition 3.16]{Balmer:2016} calls these objects `nil-faithful'). 
    \item For any noncommutative scheme $(X,\mathcal{A})$ the natural tensor action
     \[
       (-\otimes^{\mathbb{L}}_{\mathcal{O}_X}-) \colon D(X)\times D(\mathcal{A})\to D(\mathcal{A})
    \]
    gives a `projection formula' for `strict' morphism (under some flatness assumptions), see \Cref{sec:lifting}.
\end{itemize}
Exploiting these two facts yields the result.
This is the same approach taken in \cite{Aoki:2021}, our main observation is that one can leverage the action of the central scheme $X$ on the noncommutative scheme $(X,\mathcal{A})$ to extend this result to the noncommutative setting. 

Furthermore, for the \'{e}tale topology (see \cite[\href{https://stacks.math.columbia.edu/tag/0214}{Tag 0214}]{stacks-project} for definitions) we obtain a stronger statement---the existence of bounded coherent strong $\oplus$-generators is a local question.
\begin{introthm}[see \Cref{thm:etale_cover_strong_oplus_bounded_coherent_exist}]\label{introthm:nc_local_to_global_etale_big}
    Let $X$ be a separated Noetherian scheme, $\mathcal{A}$ a coherent $\mathcal{O}_X$-algebra and fix an \'{e}tale covering $\{ f_i \colon X_i \to X \}_i$.
    The following are equivalent
    \begin{enumerate}
        \item $D_{\operatorname{qc}}(\mathcal{A})$ admits a strong $\oplus$-generator with bounded and coherent cohomology,
        \item $D_{\operatorname{qc}}(f_i^\ast\mathcal{A})$ admits a strong $\oplus$-generator with bounded and coherent cohomology for each $i$.
    \end{enumerate}
\end{introthm}
To obtain this stronger result we introduce and briefly study the `diagonal dimension of a morphism' in \Cref{sec:big_rouq_and_diagonal_dim}, as a relative version of the diagonal dimension introduced in \cite{Ballard/Favero:2012}.
Bounds for strong $\oplus$-generation, using the diagonal dimension, are then obtained in \Cref{prop:pullback_generation_along_flat_separated} (and more concretely in \Cref{{cor:separated_flat_diagonal_concrete_cases}}); for this we make use of a `flat base change' for strict morphisms which is shown in \Cref{sec:lifting}.
The essential point is that a (separated) \'{e}tale morphism has diagonal dimension one, which yields the converse direction in \Cref{introthm:nc_local_to_global_etale_big}.

\subsection{Proper descent}
Our next result is a variation for proper (these are exactly those morphisms whose derived pushforwards preserve coherence) surjective morphisms, a special type of h cover; again, under suitable flatness assumptions.
It is a noncommutative generalization of \cite[Theorem C]{Dey/Lank:2024}.

\begin{introthm}[see \Cref{thm:nc_proper_descent}]\label{introthm:nc_proper_descent}
    Let $X$ be a Noetherian scheme and $\mathcal{A}$ be a coherent $\mathcal{O}_X$-algebra. 
    If $f$ is a proper surjective morphism of Noetherian schemes, and either $\mathcal{A}$ is flat over $X$ or $f$ is flat, then there exists an integer $n\geq 0$ such that $D^b_{\operatorname{coh}}(\mathcal{A}) = \langle \mathbb{R}f_\ast D^b_{\operatorname{coh}}(f^\ast \mathcal{A}) \rangle_n$.
\end{introthm}

A corollary of \Cref{introthm:nc_proper_descent} is that we can produce explicit bounds on the Rouquier dimension of $D^b_{\operatorname{coh}}(\mathcal{A})$ in terms of its pullback under a proper morphism with target the underlying scheme, again under the suitable flatness assumptions.
We refrain from spelling this bound out here but refer to \Cref{cor:rouquier_bounds_from_derived_pushforward} below (and see also \Cref{ex:bounds_Dbcoh1,ex:bounds_Dbcoh2} for concrete examples).

\subsection{Stalk-local-to-global}
Another strategy for showing the existence of (strong) generators in bounded derived categories of schemes is through leveraging stalk-local information. 
There are various situations where it is more advantageous to utilize information that is known at the level of stalks as opposed to an affine open cover. 
It is worthwhile to observe the following, which is a global version of \cite[Theorem 3.6]{Letz:2021} and a noncommutative variation of \cite[Theorem 1.7]{BILMP:2023}.

\begin{introthm}[see  \Cref{thm:local_global_stalk_generation}]\label{introth:local_global_stalk_generation}
    Let $X$ be a Noetherian scheme and $\mathcal{A}$ is a coherent $\mathcal{O}_X$-algebra. 
    Fix $E$, $G\in D^b_{\operatorname{coh}}(\mathcal{A})$.
    Suppose $E_x\in\langle G_x \rangle\subseteq D^b_{\operatorname{coh}}(\mathcal{A}_x)$ for all $x\in X$, then there exists a perfect complex $Q$ over $X$ such that $E\in \langle Q \otimes^{\mathbb{L}}_{\mathcal{O}_X} G \rangle\subseteq D^b_{\operatorname{coh}}(\mathcal{A})$. 
    In particular, we may choose $Q\in\operatorname{Perf}(X)$ to be any classical generator.
\end{introthm}

Thus similar stalk-local-to-global arguments exist in our setting.
The proof makes use of the tensor action and relies on the global-to-local principle from \cite{Stevenson:2013}.

%%%%%%%%%%%%%%%%%%%%%%%%%%%%%%%%
\subsection{Applications to Azumaya algebras}
\label{sec:intro_applications_azumaya}
%%%%%%%%%%%%%%%%%%%%%%%%%%%%%%%%

The above allows for some nice applications regarding generation for \textit{Azumaya algebras} over Noetherian schemes, which we very briefly state.
Recall that these are coherent algebras that are \'{e}tale (or even fppf) locally matrix algebras.
Our next result exploits \'etale and proper descent to show that the existence of bounded coherent generators for Azumaya algebras is intimately tied to the existence of those for their center.

\begin{introthm}[see \Cref{thm:azumaya_strong_generator_iff_center}]
    Let $X$ be a separated Noetherian scheme and $\mathcal{A}$ an Azumaya algebra over $X$. 
    Then $D_{\operatorname{qc}}(\mathcal{A})$ admits a strong $\oplus$-generator with bounded coherent cohomology if and only if $D_{\operatorname{qc}}(X)$ has such.
    Moreover, assuming $\mathcal{A}$ can be split by a finite \'{e}tale morphism, we have $\dim D^b_{\operatorname{coh}}(\mathcal{A})=\dim D^b_{\operatorname{coh}}(X)$.
\end{introthm}

This immediately yields a this yields the version of Aoki \cite[Main Theorem]{Aoki:2021} for Azumaya algebras, see \Cref{cor:nc_aoki}.
In addition, using a `d\'{e}vissage along closed integral subschemes' argument we show, in \Cref{prop:lifting_strong_gen_from_centre_for_azumaya}, that, for any Azumaya algebra $\mathcal{A}$ over a scheme $X$, every classical generator $G\in D^b_{\operatorname{coh}}(X)$ yields a classical generator $G\otimes_{\mathcal{O}_X} \mathcal{A}\in D^b_{\operatorname{coh}}(\mathcal{A})$.
Finally, for our last result, \Cref{prop:countable_rouq_dim_derived_splinter}, we extend a result of Olander \cite[Theorem 4]{Olander:2023} to Azumaya algebras over mildly singular schemes (derived splinters).

%%%%%%%%%%%%%%%%%%%%%%%%%%%%%%%%
\subsection{A comment of philosophical nature}
\label{sec:PhD}
%%%%%%%%%%%%%%%%%%%%%%%%%%%%%%%%

To end, let us say a word concerning our `mild noncommutative spaces'.
Even though we state our results for noncommutative schemes, everything we do also, of course, is true for (commutative) schemes. 
Even in the commutative setup some of the results give new perspectives, e.g.\ \'{e}tale descent for strong generation.
Our main reason for our set-up is that it is a natural framework in which one can apply geometric techniques such as descent.
There is no reason to restrict oneself to only work with `commutative spaces'---thinking of noncommutative schemes as geometric objects is a valid thing to do.

%%%%%%%%%%%%%%%%%%%%%%%%%%%%%%%%
\subsection{Notation}
\label{sec:intro_notation}
%%%%%%%%%%%%%%%%%%%%%%%%%%%%%%%%

Let $X$ be a scheme and $\mathcal{A}$ be a quasi-coherent $\mathcal{O}_X$-algebra.
Denote by
\begin{enumerate}
    \item $D(\mathcal{A}):=D(\operatorname{Mod}(\mathcal{A}))$ the derived category of (right) $\mathcal{A}$-modules,
    \item $D_{\operatorname{qc}}(\mathcal{A})$ the full subcategory of $D(\mathcal{A})$ consisting of complexes with quasi-coherent cohomology,
    \item $D_{\operatorname{coh}}^b(\mathcal{A})$ the full subcategory of $D_{\operatorname{qc}}(\mathcal{A})$ consisting of complexes with bounded and coherent cohomology,
    \item $\operatorname{Perf}(\mathcal{A})$ the full subcategory of $D^b_{\operatorname{coh}}(\mathcal{A})$ consisting of perfect complexes.
\end{enumerate}
Note that the special case of $\mathcal{A}=\mathcal{O}_X$ resorts us back to the familiar constructions for schemes; we write $D(X)$ instead of $D(\mathcal{O}_X)$ etc. If $X$ is affine, then we might at times abuse notation and write $D(A):=D_{\operatorname{qc}}(\mathcal{A})$ where $A:=H^0(X,\mathcal{A})$ are the global sections; similar conventions occur for the other categories.
Moreover, when $A$ is coherent we write $\operatorname{mod}$(:=the finitely presented modules) instead of $\operatorname{coh}$.

\begin{ack}
    Timothy De Deyn was supported under the ERC Consolidator Grant 101001227 (MMiMMa). Pat Lank was supported by the National
    Science Foundation under Grant No. DMS-2302263. Both Pat Lank and Kabeer Manali Rahul were supported under the ERC Advanced Grant 101095900-TriCatApp. Additionally, Kabeer Manali Rahul was a recipient of an Australian Government Research Training Program Scholarship, and thanks Universit\`{a} degli Studi di Milano for their hospitality during his stay there.
    Moreover, the authors thank Greg Stevenson and Jan {\v{S}}\v{t}ov\'{i}{\v{c}}ek for some helpful discussions regarding \Cref{sec:big_vs_small_elephant}.
\end{ack}

%%%%%%%%%%%%%%%%%%%%%%%%%%%%%%%%
\section{Preliminaries}
\label{sec:preliminaries}
%%%%%%%%%%%%%%%%%%%%%%%%%%%%%%%%

This section recollects some known concepts pertaining to generation in triangulated categories, provides some background for quasi-coherent algebras over schemes and gives some definitions regarding the h topology.

%%%%%%%%%%%%%%%%%%%%%%%%%%%%%%%%
\subsection{Generation}
\label{sec:generation}
%%%%%%%%%%%%%%%%%%%%%%%%%%%%%%%%

We recall some terminology and notation introduced in \cite{Bondal/VandenBergh:2003,Rouquier:2008}.
Let $\mathcal{T}$ be a triangulated category.

For a (not-necessarily triangulated) full subcategory $\mathsf{S}\subseteq\mathsf{T}$ let\footnote{It is important to note that our notation for $\operatorname{add}$ (and $\operatorname{Add}$ defined further along) differs from those in the literature. For instance, we closely follow \cite{Bondal/VandenBergh:2003}, but we also allow $\operatorname{add}$ and $\operatorname{Add}$ to be closed under direct summands, which is again slightly different from \cite{Neeman:2021b}. The primary `difference' between our definitions, \cite{Bondal/VandenBergh:2003} and \cite{Neeman:2021b} is whether or not one allows $\operatorname{add}$ and $\operatorname{Add}$ to be closed under direct summands and/or shifts; see also \cite[Remark 1.2]{Neeman:2021b}.} $\operatorname{add}\mathcal{S}$ denote the smallest strictly full subcategory of $\mathcal{T}$ containing $\mathcal{S}$ which is closed under shifts, finite coproducts and direct summands; inductively define 
\begin{displaymath}
    \langle\mathcal{S}\rangle_n :=
    \begin{cases}
        \operatorname{add}\varnothing & \text{if }n=0, \\
        \operatorname{add}\mathcal{S} & \text{if }n=1, \\
        \operatorname{add}\{ \operatorname{cone}\phi \mid \phi \in \operatorname{Hom}(\langle \mathcal{S} \rangle_{n-1}, \langle \mathcal{S} \rangle_1) \} & \text{if }n>1.
    \end{cases}
\end{displaymath}
Moreover, let $\langle\mathcal{S}\rangle:=\cup_{n\geq 0} \langle\mathcal{S}\rangle_n$, this is the smallest \textbf{thick}, i.e.\ closed under direct summands, triangulated subcategory containing $\mathcal{S}$.

If $\mathcal{T}$ admits small coproducts, define $\operatorname{Add}\mathcal{S}$ to be the smallest strictly full subcategory of $\mathcal{T}$ containing $\mathcal{S}$ which is closed under shifts, small coproducts and direct summands; again, inductively define 
\begin{displaymath}
    \overline{\langle\mathcal{S}\rangle}_n :=
    \begin{cases}
        \operatorname{Add}\varnothing & \text{if }n=0, \\
        \operatorname{Add}\mathcal{S} & \text{if }n=1, \\
        \operatorname{Add}\{ \operatorname{cone}\phi \mid \phi \in \operatorname{Hom}(\overline{\langle \mathcal{S} \rangle}_{n-1}, \overline{\langle \mathcal{S} \rangle}_1) \} & \text{if }n>1.
    \end{cases}
\end{displaymath}
Whenever $\mathcal{S}=\{G\}$ consist of a single object, we leave out the brackets simply writing $\langle G\rangle$ etc.

An object $G\in \mathcal{T}$ is called a \textbf{classical generator} if $\mathcal{T}=\langle G \rangle$ and a \textbf{strong generator} if $\mathcal{T} = \langle G \rangle_{n+1} $ for some $n\geq 0$.
Moreover, the smallest such $n$ (amongst all possible objects $G$) is called the \textbf{Rouquier dimension} of $\mathcal{T}$ and is denoted $\dim \mathcal{T}$.
If, in addition, $\mathcal{T}$ has small coproducts then $G$  is called  a \textbf{strong $\oplus$-generator} if $\mathcal{T} = \overline{\langle G \rangle}_{n+1}$ for some $n\geq 0$.

Lastly, unless otherwise specified, we adopt the following convention: for a triangulated functor $F\colon \mathcal{T} \to \mathcal{T}^\prime$, $F(\mathcal{T})$ denotes the essential image, i.e.\ the smallest strictly full subcategory containing $\{ F(T)\in\mathcal{T}' \mid T\in\mathcal{T} \}$.

%%%%%%%%%%%%%%%%%%%%%%%%%%%%%%%%
\subsection{Noncommutative schemes}
\label{sec:nc_scheme}
%%%%%%%%%%%%%%%%%%%%%%%%%%%%%%%%

The primary source of reference for what follows below is \cite[$\S 3$]{DeDeyn/Lank/ManaliRahul:2024}. We refer to loc.\ cit.\ for more information as we only briefly recall a few things relevant for this work.

\begin{definition}\label{def:nc_scheme}
    \hfill
    \begin{enumerate}
        \item A \textbf{(mild) noncommutative scheme} is a pair $(X,\mathcal A)$ where $X$ is a scheme and $\mathcal A$ is a quasi-coherent $\mathcal{O}_X$-algebra.
        \item A \textbf{morphism of noncommutative schemes} $(f, \varphi)\colon(Y, \mathcal B)\to (X,\mathcal A)$ consists of a morphism of schemes $f\colon Y\to X$ and a morphism $\varphi\colon f^{-1}\mathcal A\to\mathcal B$ of sheaves of $f^{-1}\mathcal O_{X}$-algebras. 
        \item A noncommutative scheme $(X,\mathcal A)$ is called 
        \begin{enumerate}
            \item \textbf{affine} if $X=\operatorname{Spec}(R)$ for some commutative ring $R$,
            \item \textbf{coherent} if $\mathcal A$ is coherent as an $\mathcal{O}_X$-module,
            \item \textbf{integral} if its underlying scheme $X$ is integral,
            \item \textbf{Noetherian} if its underlying scheme $X$ is Noetherian.
        \end{enumerate}
        In addition, we will use `topological' terminology (i.e quasi-compact, separated, etc.) for a nc.\ scheme $(X,\mathcal{A})$ if the underlying scheme $X$ enjoys such properties.
    \end{enumerate}
\end{definition}

In this work we will be mainly focused on coherent Noetherian nc.\ schemes, and so they deserve a special (shorter) and more appropriate name.
Recall that a \textbf{Noether algebra} is a Noetherian (potentially noncommutative) ring which is finite (as module) over its center.
The coherent Noetherian affine nc.\ schemes are exactly the Noether algebras.
For this reason we adopt the following terminology.

\begin{definition}
    A noncommutative scheme $(X,\mathcal{A})$ is called a \textbf{Noether noncommutative scheme} if it is coherent and Noetherian.
\end{definition}

\begin{example}
    \hfill
    \begin{enumerate}
        \item Any scheme $X$ yields a nc.\ scheme $(X,\mathcal{O}_X)$ which we will simply keep writing as just $X$.
        \item For any nc.\ scheme $(X,\mathcal{A})$ the morphism $\mathcal{O}_X\to \mathcal{A}$ induces a morphism of noncommutative schemes $\pi\colon(X,\mathcal{A})\to X$ which we refer to as the \textbf{structure morphism}.
        \item Let $X$ be a scheme. If $\mathcal{E}$ is a locally free $\mathcal{O}_X$-module of finite rank, then $\operatorname{\mathcal{E}\! \mathit{nd}}(\mathcal{E})$ (the sheaf Hom) is a finite locally free (hence flat) $\mathcal{O}_X$-algebra. 
        \item An \textbf{Azumaya algebra} on a Noetherian scheme $X$ is a coherent sheaf of algebras $\mathcal{A}$ over $X$ that is \'{e}tale (or even fppf) locally split in the sense that it is a matrix algebra; see e.g.\ \cite{Milne:1980} for more characterisations.
        More generally, we say an Azumaya algebra \textbf{splits} when it is as in the previous item (i.e.\ it represents the trivial class in the Brauer group).
        We call any such pair $(X,\mathcal{A})$ an \textbf{Azumaya scheme} (so in particular we assume $X$ is Noetherian). 
        The special case where $X$ is a variety occurs in \cite[Appendix D]{Kuznetsov:2006}, where their associated derived categories are studied.
    \end{enumerate}
\end{example}

For a nc.\ scheme $(X,\mathcal A)$ we let $\operatorname{Mod}(X,\mathcal{A})$, $\operatorname{Qcoh}(X,\mathcal{A})$ and $\operatorname{coh}(X,\mathcal{A})$ denote the usual categories of respectively (right) $\mathcal{A}$-modules, quasi-coherent $\mathcal{A}$-modules and coherent $\mathcal{A}$-modules\footnote{Observe that quasi-coherence as $\mathcal{A}$-module can be checked at the level of the underlying $\mathcal{O}_X$-modules because $\mathcal A$ is a quasi-coherent $\mathcal{O}_X$-algebra. The same is true for coherence, granted that $\mathcal{A}$ is coherent.}.
Of course, we also consider their appropriate derived categories.
We define $D(X, \mathcal A):= D(\operatorname{Mod}(X,\mathcal{A}))$ the derived category of $\mathcal A$-modules, $D_{\operatorname{qc}}(X, \mathcal{A})$ the strictly full subcategory consisting of objects with quasi-coherent cohomology and $ D^b_{\operatorname{coh}}(X, \mathcal{A})$ is the strictly full subcategory consisting of objects with bounded and coherent cohomology\footnote{
Note that the natural functor from $D(\operatorname{Qcoh}(\mathcal A))$ to $D_{\operatorname{qc}}(\mathcal A)$ is an equivalence when $(X,\mathcal A)$ is quasi-compact and separated. Additionally, the natural functor from $D^-(\operatorname{coh}(\mathcal A))$ to $D^-_{\operatorname{coh}}(\operatorname{Qcoh}(\mathcal A))$ is an equivalence when $(X,\mathcal A)$ is coherent and Noetherian. See \cite[Proposition 1.3]{Alonso/Jeremias/Lipman:1997} and \cite[\href{https://stacks.math.columbia.edu/tag/0FDA}{Tag 0FDA}]{stacks-project} for proofs of the commutative variants, which extend to our setting.
}.    
When clear from context we leave the $X$ out of the previous notations, simply writing $\operatorname{Mod}(\mathcal{A})$, etc.
Moreover, if $X$ is affine, then we might at times abuse notation and write $D(A):=D_{\operatorname{qc}}(\mathcal{A})$ where $A:=H^0(X,\mathcal{A})$ are the global sections.
In addition, let us note that a morphism of nc.\ schemes $f \colon (Y,\mathcal B) \to (X, \mathcal A)$ induces the usual pullback-pushforward adjunction between the appropriate categories (under suitable hypotheses).
Derived functors can be defined and computed per usual, see \cite{Spaltenstein:1988}. 
A few references are \cite[Appendix D]{Kuznetsov:2006}, \cite[Proposition 3.12]{Burban/Drozd/Gavran:2017}, or \cite[$\S 3$]{DeDeyn/Lank/ManaliRahul:2024}.

\begin{construction}\label{con:breve}
    Let $(X,\mathcal{A})$ be a nc.\ scheme and let $f\colon Y \to X$ be a morphism of schemes. 
    As $f^\ast \mathcal{A}$ is a quasi-coherent $\mathcal{O}_Y$-algebra, we obtain a nc.\ scheme $(Y,f^\ast \mathcal{A})$ and an induced strict\footnote{A morphism $f\colon (Y,\mathcal{B}) \to (X, \mathcal{A})$ is called \textbf{strict} when $f^\ast\mathcal{A}=\mathcal{B}$.} morphism $(Y,f^\ast \mathcal{A})\to (X, \mathcal{A})$  by taking $\varphi\colon f^{-1}\mathcal A\to f^\ast\mathcal{A}=f^{-1}\mathcal A\otimes_{f^{-1}\mathcal O_X}\mathcal O_Y$ to be the obvious map.
    When confusion can occur we denote this induced morphism by $\Breve{f}$.
    However, as this morphism is compatible with the structure morphisms, we often abuse notation and simply denote the resulting morphism of nc.\ schemes again by $f$.
    Moreover, it is often notationally convenient to denote $f^\ast\mathcal{A}$ by $\mathcal{A}_Y$  when no confusion can occur (i.e.\ there is only one map $Y\to X$ in play).
    In this way one obtains a functor $\operatorname{Sch}/X \to \operatorname{nc.Sch}/(X,\mathcal{A}),~Y\mapsto  (Y,\mathcal{A}_Y)$.
\end{construction}

For examples of nc.\ schemes whose derived categories have generators in the sense of \Cref{sec:generation} we refer to \cite[Example 3.18]{DeDeyn/Lank/ManaliRahul:2024}. For future reference, let us observe the following variation of \cite[Lemma 3.9]{DeDeyn/Lank/ManaliRahul:2024} with exactly the same proof.

\begin{lemma}\label{lem:nc_bounded_above}
    Let $f \colon (Y,\mathcal B) \to (X, \mathcal A)$ be a proper morphism of Noether nc.\ schemes. 
    The derived pushforward restricts to an exact functor $\mathbb{R}f_\ast \colon D^-_{\operatorname{coh}}(\mathcal{B}) \to D^-_{\operatorname{coh}}( \mathcal{A})$.
\end{lemma}

%%%%%%%%%%%%%%%%%%%%%%%%%%%%%%%%
\subsection{h topology}
\label{sec:h-top}
%%%%%%%%%%%%%%%%%%%%%%%%%%%%%%%%

We dabble with the h topology, introduced in \cite{Voevodsky:1996}, in this work.
Therefore, we briefly recall the main definitions so everyone is on the same page. 
For non-Noetherian schemes definitions diverge a bit, however, we will only deal with Noetherian schemes, so we restrict to those for ease.

\begin{definition}[{see \cite[\href{https://stacks.math.columbia.edu/tag/040H}{Tag 040H}]{stacks-project}}]
    Let  $f\colon X\to Y$ be a morphism of schemes.
    \begin{enumerate}
        \item We say $f$ is \textbf{submersive} if the underlying continuous map of topological spaces is submersive.
        This means that the continuous map is surjective and $U\subseteq Y$ is an open or closed subset if and only if $f^{-1}(U)\subseteq X$ is so.
        \item We say $f$ is \textbf{universally submersive} if for every morphism of schemes $Y^\prime \to Y$ the base change $Y^\prime\times_Y X \to Y$ is submersive.
    \end{enumerate}
\end{definition}

\begin{definition}\label{def:h_topology}
        Let $X$ be a Noetherian scheme. A finite family of finite type morphisms
        % \footnote{Note that the $X_i$ are automatically Noetherian, as rings of finite type over Noetherian rings are Noetherian and finite type morphisms are quasi-compact.} 
        $\{f_i \colon X_i \to X\}_{i\in I}$ is called a \textbf{(finite) h  covering} if $\amalg_{i \in I} X_ i \to X$ is universally submersive. When the family consists of a single morphism $\{f:Y\to X\}$ we call $f$ an \textbf{h  cover}.
\end{definition}

This is consistent with \cite[Definition 2.7]{Bhatt/Scholze:2017} taking into account  \cite[\href{https://stacks.math.columbia.edu/tag/0ETS}{Tag 0ETS}]{stacks-project} and \cite[\href{https://stacks.math.columbia.edu/tag/0ETT}{Tag 0ETT}]{stacks-project}.

%%%%%%%%%%%%%%%%%%%%%%%%%%%%%%%%
\section{Lifting along strict morphisms}
\label{sec:lifting}
%%%%%%%%%%%%%%%%%%%%%%%%%%%%%%%%

In this section we `lift' some statements from (commutative) algebraic geometry to our nc.\ setting.
The easiest way to achieve this is to work with strict morphisms of nc.\ schemes over some fixed nc.\ base scheme as in \Cref{con:breve}.
In particular, we show analogues of the `projection formula' and `flat base change' in our set-up for strict morphisms.
These allow us to lift generation statements from the commutative world to the noncommutative world in \Cref{sec:descent_conditions}.

A crucial ingredient for obtaining a projection formula is to have a suitable `tensor' to formulate the formula with (nc.\ schemes generally no longer have a monoial structure). 
To this end, we consider the action of the underlying central scheme on the nc.\ scheme, i.e.\ for any nc.\ scheme $(X,\mathcal{A})$ we consider the following action
\begin{equation}\label{eq:tensor-action}
   (-\otimes^{\mathbb{L}}_{\mathcal{O}_X}-) \colon D(X)\times D(\mathcal{A})\to D(\mathcal{A})
\end{equation}
and its suitable restrictions. 
We refer to \cite{Stevenson:2013} for the basics of tensor triangulated categories acting on other triangulated categories. 

\begin{proposition}\label{prop:strict_form_nc_projection_formula}
    Let $(X,\mathcal{A})$ be a nc.\ scheme and let $f\colon Y\to X$ be a quasi-compact and quasi-separated morphism of schemes. 
    % Consider the induced morphism of nc.\ schemes $\breve{f}\colon (Y,f^\ast\mathcal{A})\to (X,\mathcal{A})$ from Construction \Cref{con:breve}.
    Assume either $\mathcal{A}$ is flat over $X$ or $f$ is a flat morphism. 
    There is a canonical natural isomorphism
    \begin{equation}\label{eq: ncag proj formula}
        \mathbb{R}f_\ast E \otimes_{\mathcal{O}_X}^{\mathbb{L}} F \xrightarrow{\sim} \mathbb{R}\breve{f}_\ast (E \otimes_{\mathcal{O}_Y}^{\mathbb{L}} \mathbb{L}\breve{f}^\ast F)
   \quad \text{ (in $D_{\operatorname{qc}}(\mathcal{A})$)}     
    \end{equation}
    for all $E$ in $D_{\operatorname{qc}}(X)$ and $F$ in $D_{\operatorname{qc}}(\mathcal{A})$.
\end{proposition}

\begin{proof}
    That the morphism in \eqref{eq: ncag proj formula} actually exists follows from the pullbacks being compatible with the tensor actions. 
    We have
    \begin{equation}\label{eq:'monoidal'}
    \mathbb{L}\breve{f}^\ast(\mathbb{R} f_\ast E\otimes_{\mathcal{O}_X}^{\mathbb{L}} F)= \mathbb{L}{f}^\ast\mathbb{R} f_\ast E\otimes_{\mathcal{O}_Y}^{\mathbb{L}} \mathbb{L}\breve{f}^\ast F .
    \end{equation}
    Thus, one obtains the morphism, as usual, by adjunction from 
    \[
        \mathbb{L}\breve{f}^\ast(\mathbb{R} f_\ast E\otimes_{\mathcal{O}_X}^{\mathbb{L}} F) = \mathbb{L}{f}^\ast\mathbb{R} f_\ast E\otimes_{\mathcal{O}_Y}^{\mathbb{L}} \mathbb{L}\breve{f}^\ast F  
        \xrightarrow{\text{ counit }}  E\otimes_{\mathcal{O}_Y}^{\mathbb{L}} \mathbb{L}\breve{f}^\ast F .
    \]

    To show it is an isomorphism, it suffices to show this after applying the structure morphism $\pi_{X,\ast}$ as this reflects (quasi-)isomorphisms.
    So, it suffices to show
    \[
        \mathbb{R}f_\ast E \otimes_{\mathcal{O}_X}^{\mathbb{L}} \pi_{X,\ast}F = \mathbb{R}{f}_\ast (E \otimes_{\mathcal{O}_Y}^{\mathbb{L}} \pi_{Y,\ast}\mathbb{L}\breve{f}^\ast F)\quad\text{(in $D_{\operatorname{qc}}(X)$).}
    \]
    However, note that $\pi_{Y,\ast}\mathbb{L}\breve{f}^\ast F=\mathbb{L}{f}^\ast\pi_{X,\ast}F$ by either of the flatness assumptions (e.g.\ reduce to checking this for $F=\mathcal{A})$, hence we reduce to the usual projection formula, see e.g.\ \cite[\href{https://stacks.math.columbia.edu/tag/08EU}{Tag 08EU}]{stacks-project}. 
\end{proof}

A variant, with proof along the same lines, is the following.

\begin{proposition}\label{prop:tensor_nc_projection_formula}
    Let $(X,\mathcal{A})$ be a nc.\ scheme and let $f\colon Y\to X$ be a quasi-compact quasi-separated morphism of schemes. There is a canonical natural isomorphism
    \[
        F \otimes^{\mathbb{L}}_{\mathcal{O}_X} \mathbb{R}\breve{f}_\ast E \xrightarrow{\sim} \mathbb{R}\breve{f}_\ast (\mathbb{L} f^\ast F \otimes^{\mathbb{L}}_{\mathcal{O}_Y} E)
   \quad \text{ (in $D_{\operatorname{qc}}(\mathcal{A})$)}     
    \]
    for all $F$ in $D_{\operatorname{qc}}(X)$ and $E$ in $D_{\operatorname{qc}}(\mathcal{A}_Y)$.
\end{proposition}

To end the section we note the following version of nc.\ flat base change.

\begin{proposition}\label{prop:nc_base_change}
    Let $(X,\mathcal{A})$ be a nc.\ scheme, $g\colon X^\prime \to X$ be a flat morphism of schemes, and $f\colon Y \to X$ be a quasi-compact quasi-separated morphism of schemes. Consider the base change diagram
    \begin{displaymath}
        \begin{tikzcd}[ampersand replacement=\&]
            {Y^\prime} \& Y \\
            {X^\prime} \& X\rlap{ .}
            \arrow["{g^\prime}", from=1-1, to=1-2]
            \arrow["{f^\prime}"', from=1-1, to=2-1]
            \arrow["\lrcorner"{anchor=center, pos=0.125}, draw=none, from=1-1, to=2-2]
            \arrow["f", from=1-2, to=2-2]
            \arrow["g"', from=2-1, to=2-2]
        \end{tikzcd}
    \end{displaymath}
    There is a canonical natural isomorphism
    \begin{equation}\label{eq: nc_flat}
        \mathbb{L} \Breve{g}^\ast \mathbb{R} \Breve{f}_\ast E \xrightarrow{\sim} \mathbb{R}\Breve{f^\prime}_\ast \mathbb{L} \Breve{g^\prime}^\ast E \quad \text{ (in $D_{\operatorname{qc}}(\mathcal{A}_{X'})$)} 
    \end{equation}
    for all $E$ in $D_{\operatorname{qc}}(\mathcal{A})$.
\end{proposition}

\begin{proof}
    That the morphism in \eqref{eq: nc_flat} exists follows by the same argument as in the commutative case, see e.g.\ \cite[\href{https://stacks.math.columbia.edu/tag/08HY}{Tag 08HY}]{stacks-project}. 
    Again, as in \Cref{prop:strict_form_nc_projection_formula}, the proof goes by reducing\footnote{To see that $\pi_\ast$ of \eqref{eq: nc_flat} yields the commutative flat base change map, it is easiest to pick $K$-flasque resolutions by which one can reduce to checking this with non-derived functors.} to the `commutative' flat base change formula of the underlying schemes \cite[\href{https://stacks.math.columbia.edu/tag/08IB}{Tag 08IB}]{stacks-project}.
\end{proof}

%%%%%%%%%%%%%%%%%%%%%%%%%%%%%%%%
\section{Big Rouquier dimension and diagonal dimension}
\label{sec:big_rouq_and_diagonal_dim}
%%%%%%%%%%%%%%%%%%%%%%%%%%%%%%%%

It is convenient to consider relative versions of Rouquier dimension, where one fixes a subcategory where the generating object is forced to live.
The following is the `big' version, where we allow for arbitrary small coproducts.

\begin{definition}\label{def:relative_dimension}
    Let $\mathcal{T}$ be a triangulated category with coproducts and $\mathcal{S}$ be a subcategory. We define the \textbf{$\oplus$-dimension of $\mathcal{T}$ with respect to $\mathcal{S}$}, denoted $\dim_\oplus (\mathcal{T},\mathcal{S})$, to be the smallest integer $n\geq 0$ such that $\mathcal{T} = \overline{\langle G \rangle}_{n+1}$ for some object $G$ of $\mathcal{S}$ and to be $\infty$ when no such $n$ exists. 
\end{definition}

Next, we extend \cite[Definition 2.15]{Ballard/Favero:2012} to a relative setting.
To state it, first recall the following commonly used notation.
For a morphism of schemes\footnote{As $\times_X$ looks relatively `unsightly' we will switch to $S$ for a base scheme instead of $X$ at times.} $s \colon X\to S$ with the natural projections denoted $p_i\colon X\times_S X \to X$  we define $G_1 \boxtimes_S G_2:=\mathbb{L} p^\ast_1 G_1 \otimes^{\mathbb{L}}_{\mathcal{O}_{X\times_S X}} \mathbb{L} p^\ast_2 G_2 \in D_{\operatorname{qc}}(X\times_S X)$ for any two objects $G_1$ and $G_2$ in $D_{\operatorname{qc}}(X)$. 

\begin{definition}\label{def:diagonal_dimension}
    Let $s \colon X\to S$ be a morphism of schemes and $\mathcal{S}$ be a subcategory of $D_{\operatorname{qc}}(X)$. 
    The \textbf{diagonal dimension of $s$ with respect to $\mathcal{S}$}, denoted $\dim_\Delta( s,\mathcal{S})$, is defined to be the smallest integer $n\geq 0$ for which $\mathbb{R}\Delta_\ast \mathcal{O}_X$ belongs to $\langle G_1 \boxtimes_S G_2 \rangle_{n}$ for some $G_1$ and $G_2$ in $\mathcal{S}$.
    It is defined to be $\infty$ when no such $n$ exists.
\end{definition}

\begin{lemma}\label{lem:alt_Delta_dim}
    With notation as in \Cref{def:diagonal_dimension}.
    \begin{enumerate}
        \item If $\mathcal{S}$ is closed under finite coproducts, then $\dim_\Delta (s,\mathcal{S})$ is equal to the smallest integer $n$ (when it exists) for which $\mathbb{R}\Delta_\ast \mathcal{O}_X$ belongs to $\langle G \boxtimes_S G \rangle_{n}$ for $G$ in $\mathcal{S}$.
        \item If $\mathcal{S}^\prime$ is contained in $\mathcal{S}$, then $\dim_\Delta (s,\mathcal{S}) \leq \dim_\Delta (s,\mathcal{S}^\prime)$.
    \end{enumerate}
\end{lemma}

\begin{proof}
    The second claim follows immediately from the definition, so we check the first claim. 
    We set $m:=\dim_\Delta (s, \mathcal{S})$.
    Clearly $m\leq n$. 
    For the opposite inequality, observe that if $\mathbb{R}\Delta_\ast \mathcal{O}_X$ belongs to $\langle G_1 \boxtimes_S G_2 \rangle_{m+1}$ with $G_1$ and $G_2$ objects of $\mathcal{S}$.
    Then, $G_1\boxtimes_S G_2$ belongs to $\langle G \boxtimes_S G \rangle_1$ where $G:=G_1 \oplus G_2\in \mathcal{S}$. 
    Hence, if $\mathbb{R}\Delta_\ast \mathcal{O}_X$ belongs to $\langle G_1 \boxtimes_S G_2 \rangle_{m}$, then also to $\langle G \boxtimes_S G \rangle_{m}$, showing that $n\leq m$.
\end{proof}

Let us give some examples of bounds on diagonal dimension.

\begin{example}
    Suppose $X$ is a variety over a field $\mathsf{k}$ and let $s \colon X \to \operatorname{Spec}(\mathsf{k})$ denote the natural map. 
    
    If the ground field $\mathsf{k}$ is perfect, then, by \cite[Theorem 4.18]{Elagin/Lunts/Schnurer:2020}, $G\boxtimes_{\mathsf{k}} G$ is a strong generator for $D^b_{\operatorname{coh}}(X\times_k X)$ whenever $G$ is one for $D^b_{\operatorname{coh}}(X)$. In particular, this ensures that $A\boxtimes_{\mathsf{k}} B$ lies in $D^b_{\operatorname{coh}}(X\times_{\mathsf{k}} X)$ for $A$ and $B$ objects in $D^b_{\operatorname{coh}}(X)$.
    In particular, we have $\dim_\Delta(s,D^b_{\operatorname{coh}}(X))=\dim_\Delta(X)$ where the latter is the diagonal dimension as defined in \cite[Definition 2.15]{Ballard/Favero:2012}.     

    When the field $\mathsf{k}$ is not perfect one could/should pick $\mathcal{S}$ as a subcategory of $D^b_{\operatorname{coh}}(X\times_{\mathsf{k}}X)$ in \Cref{def:diagonal_dimension} in order to obtain a relative version of loc.\ cit.\ But, of course, when $X$ is regular one also has $\dim_\Delta(s,D^b_{\operatorname{coh}}(X))=\dim_\Delta(X)$.
\end{example}

\begin{example}\label{ex:diagdim_etale}
    Let $f\colon X \to S$ be a separated \'{e}tale morphism of schemes. Then $\dim_\Delta(f,\operatorname{Perf}(X))=1$ as the diagonal morphism $\Delta\colon X \to X\times_S X$ is both an open and closed immersion. In fact, $\Delta$ is the inclusion of $X$ in $X\times_S X \cong X \sqcup Z$ for some scheme $Z$. Consequently, $\mathbb{R}\Delta_\ast \mathcal{O}_X=\Delta_\ast \mathcal{O}_X$ is a direct summand of $\mathcal{O}_{X\times_S X}$ and so belongs to $\langle \mathcal{O}_X \boxtimes_S \mathcal{O}_X \rangle_{1}$ (note that $\mathcal{O}_X \boxtimes_S \mathcal{O}_X = \mathcal{O}_{X\times_S X}$).
\end{example}

\begin{example}
    Let $S$ be a scheme. Suppose $\mathcal{E}$ is a locally free sheaf of rank $r+1$ over $S$. Denote the associated projective bundle by $X := \mathbb{P}(\mathcal{E})$ with $p \colon X \to S$ the associated morphism. Then $\operatorname{dim}_{\Delta}( p, \operatorname{Perf}(X)) \leq r$, as can be shown by explicitly constructing an exact sequence
    \begin{displaymath}
        0 \to \mathcal{O}_X(-n) \boxtimes \Omega^n(n) \to \cdots \to \mathcal{O}_X(-1) \boxtimes \Omega(1) \to \mathcal{O}_{X \times_S X} \to \mathcal{O}_{\Delta} \to 0,
    \end{displaymath}
    similar to Be\u{\i}linson's famous result from \cite{Beilinson:1978}.
    The construction starts with the exact sequence 
    \begin{displaymath}
         \mathcal{O}_X(-1) \boxtimes \Omega(1) \xrightarrow{\Phi} \mathcal{O}_{X \times_S X} \to \mathcal{O}_{\Delta} \to 0
    \end{displaymath}
    obtained from the the Euler sequence (see \cite[Theorem 4.5.13]{Brandenburg:2014} for the proof of a very general version of such).
    The Koszul complex associated to $\Phi$ then yields the required exact sequence.
\end{example}

\begin{example}
    The following is a slight generalization/reformulation of \cite[\href{https://stacks.math.columbia.edu/tag/0FZ5}{Tag 0FZ5}]
    {stacks-project} and a part of \cite[Lemma 2.16]{Ballard/Favero:2012}.

    Let $X$ be a separated finite type scheme over some Noetherian affine base scheme $S$ and let $s\colon X\to S$ denote the natural morphism. 
    Assume $X$ has the resolution property and $X\times_S X$ is regular of finite Krull dimension. 
    Concrete occurrence of this set-up are when $S$ is regular and $s$ is smooth (see e.g.\ \cite[\href{https://stacks.math.columbia.edu/tag/036D}{Tag 036D}]{stacks-project} and \cite[\href{https://stacks.math.columbia.edu/tag/0F8A}{Tag 0F8A}]{stacks-project}); e.g.\ $X$ is some smooth variety over a field. 
    Then, by \cite[\href{https://stacks.math.columbia.edu/tag/0FZ2}{Tag 0FZ2}]{stacks-project}, as $X$ has the resolution property, there exists a resolution
    \begin{displaymath}
        \cdots \to \mathcal{F}_i \boxtimes_S \mathcal{G}_i \to \cdots \to \mathcal{F}_0 \boxtimes_S \mathcal{G}_0 \to \Delta_\ast \mathcal{O}_X \to 0
    \end{displaymath}
    where the $\mathcal{F}_i$, $\mathcal{G}_i$ are finite locally free sheaves over $X$.
    Furthermore, as $X\times_S X$ is regular of finite Krull dimension, we have $\operatorname{Ext}^j_{X\times_S X} (\mathcal{F},\mathcal{G}) = 0$ for any coherent modules $\mathcal{F}$, $\mathcal{G}$ and all $j > \dim (X\times_S X)=:d$, see e.g.\ \cite[\href{https://stacks.math.columbia.edu/tag/0FZ3}{Tag 0FZ3}]{stacks-project}. 
    So, consider the following (stupid) truncation:
    \begin{displaymath}
        P:= (0\to \mathcal{F}_{d} \boxtimes_S \mathcal{G}_{d} \to \cdots \to \mathcal{F}_0 \boxtimes_R \mathcal{G}_0 \to 0).
    \end{displaymath}
    This complex only has nonzero cohomology in degrees zero (where it is $\Delta_\ast \mathcal{O}_X$) and $-d$.
    It follows from \cite[\href{https://stacks.math.columbia.edu/tag/0FZ4}{Tag 0FZ4}]{stacks-project} that $\Delta_\ast \mathcal{O}_X$ is a direct summand of $P$. 
    Let $Q:= (\oplus_{i=0}^{d} \mathcal{F}_i )\boxtimes_S (\oplus_{i=0}^{d} \mathcal{G}_i )$.
    Then $\Delta_\ast \mathcal{O}_X$ is in $\langle Q \rangle_{d}$, which implies that $\dim_\Delta(s,\operatorname{Perf}(X))\leq d=\dim X\times_S X$.
\end{example}

A morphism having finite diagonal dimension allows one to obtain generators of the source from those of the target. 

\begin{proposition}\label{prop:pullback_generation_along_flat_separated}
    Let $(S,\mathcal{A})$ be a Noether nc.\ scheme and $s\colon X \to S$ be a separated flat morphism of Noetherian schemes. Consider subcategories $\mathcal{S}\subseteq D_{\operatorname{qc}}(X)$ and $\mathcal{T}\subseteq D_{\operatorname{qc}}(\mathcal{A})$ where $\mathcal{S}$ is closed under finite sums. Then the following inequality holds:
    \begin{displaymath}
        \dim_{\oplus}(D_{\operatorname{qc}}(\mathcal{A}_X), \mathcal{S}\otimes^{\mathbb{L}}_{\mathcal{O}_{ X}} \mathbb{L} \breve{s}^\ast \mathcal{T}) \leq \dim_\Delta (s,\mathcal{S}) \cdot (\dim_{\oplus}(D_{\operatorname{qc}}( \mathcal{A}), \mathcal{T}) + 1) -1.
    \end{displaymath}
    In fact, when the right hand side of the inequality above is finite, then $G\otimes^{\mathbb{L}}_{\mathcal{O}_{ X}} \mathbb{L} \breve{s}^\ast H$ is a strong $\oplus$-generator for $D_{\operatorname{qc}}(\mathcal{A}_X)$ giving the sought inequality, where $H$ is as in \Cref{def:relative_dimension} and $G:=G_1\oplus G_2$ is as in \Cref{def:diagonal_dimension}.
\end{proposition}

\begin{proof}
    Clearly we may assume $n:=\dim_\Delta (s,\mathcal{S})$ and $m:=\dim_{\oplus}(D_{\operatorname{qc}}(\mathcal{A}),\mathcal{T})$ are finite.
    Thus, by \Cref{def:diagonal_dimension} and \Cref{lem:alt_Delta_dim}, there exists $G$ in $\mathcal{S}$ with $\Delta_\ast \mathcal{O}_X=\mathbb{R}\Delta_\ast \mathcal{O}_X\in\langle G \boxtimes_S G \rangle_n$, where $\Delta \colon X \to X\times_S X$ denotes the diagonal morphism, and $H$ is an object of $\mathcal{T}$ with $D_{\operatorname{qc}}(\mathcal{A})=\overline{\langle H\rangle}_{m+1}$.
    It suffices to show that 
    \begin{displaymath}
        D_{\operatorname{qc}}(\mathcal{A}_X)=\overline{\langle G\otimes^{\mathbb{L}} \breve{s}^\ast H\rangle}_{n(m+1)}
    \end{displaymath}
    as the inequality follows from this. Consider the pullback square:
    \begin{displaymath}
        \begin{tikzcd}[ampersand replacement=\&]
            {X\times_S X} \& X \\
            X \& S\rlap{ .}
    	\arrow["{p_2}", from=1-1, to=1-2]
    	\arrow["{p_1}"', from=1-1, to=2-1]
    	\arrow["s", from=1-2, to=2-2]
    	\arrow["s"', from=2-1, to=2-2]
     	\arrow["\lrcorner"{anchor=center, pos=0.125}, draw=none, from=1-1, to=2-2]
        \end{tikzcd}
    \end{displaymath}
    Note that all the morphisms appearing are flat. Fix an arbitrary object $E\in D_{\operatorname{qc}} ( \mathcal{A}_X)$.
    As most tensor products will be over $X\times_S X$, we simply write $\otimes$ instead of $\otimes_{\mathcal{O}_{X\times_S X}}$ for the rest of this proof.
    We have the following string of implications:
    \begin{align*}
        \Delta_\ast  \mathcal{O}_X &\in \langle p^\ast_1 G \otimes^{\mathbb{L}}p^\ast_2 G \rangle_n \\
        \implies \Delta_\ast \mathcal{O}_X \otimes^{\mathbb{L}} \Breve{p}_1^\ast E &\in \langle p^\ast_1 G \otimes^{\mathbb{L}} p^\ast_2 G \otimes^{\mathbb{L}} \Breve{p}_1^\ast E\rangle_n \\
        \implies \mathbb{R}\Breve{p}_{2,\ast} ( \Delta_\ast \mathcal{O}_X \otimes^{\mathbb{L}} \Breve{p}_1^\ast E) &\in \langle \mathbb{R}\Breve{p}_{2,\ast} (  p^\ast_1 G \otimes^{\mathbb{L}} p^\ast_2 G \otimes^\mathbb{L} \Breve{p}_1^\ast E)\rangle_n.
    \end{align*}
    Then, using the projection formula (\Cref{prop:strict_form_nc_projection_formula}) we obtain:
    \begin{displaymath}
        \begin{aligned}
            \mathbb{R}\Breve{p}_{2,\ast} ( \Delta_\ast \mathcal{O}_X \otimes^{\mathbb{L}} \Breve{p}_1^\ast E) &\cong \mathbb{R}\Breve{p}_{2,\ast} (\Breve{\Delta}_\ast \mathbb{L} \Breve{\Delta}^\ast \Breve{p}_1^\ast E) \\&\cong \mathbb{R}(\Breve{p_2} \circ \Breve{\Delta})_\ast \mathbb{L} (\Breve{p_1} \circ \Breve{\Delta})^\ast  E  \\&\cong E \qquad ( \operatorname{id}_X = \breve{p_i} \circ \breve{\Delta}).
        \end{aligned}        
    \end{displaymath}
    On the other hand, using \Cref{prop:nc_base_change,prop:tensor_nc_projection_formula}, we obtain another isomorphism:
    \begin{displaymath}
        \begin{aligned}
            \mathbb{R}\Breve{p}_{2,\ast}  ( p^\ast_1 G \otimes^{\mathbb{L}} p^\ast_2 G \otimes^\mathbb{L} \Breve{p}_1^\ast E) &\cong \mathbb{R}\Breve{p}_{2,\ast} ( p^\ast_2 G \otimes^{\mathbb{L}} (p^\ast_1 G \otimes^\mathbb{L} \Breve{p}_1^\ast E))
            \\&\cong G\otimes^\mathbb{L}_{\mathcal{O}_X} \mathbb{R}\Breve{p}_{2,\ast} ( p^\ast_1 G \otimes^\mathbb{L} \Breve{p}_1^\ast E)
            \\&\cong G\otimes^\mathbb{L}_{\mathcal{O}_X} \mathbb{R}\Breve{p}_{2,\ast} \Breve{p}_1^\ast ( G \otimes^\mathbb{L}_{\mathcal{O}_X} E)
            \\&\cong G\otimes^\mathbb{L}_{\mathcal{O}_X} \Breve{s}^\ast \mathbb{R} \Breve{s}_\ast ( G \otimes^\mathbb{L}_{\mathcal{O}_X} E).
        \end{aligned}
    \end{displaymath}
    As $\mathbb{R} \Breve{s}_\ast ( G \otimes^\mathbb{L}_{\mathcal{O}_X} E)\in\overline{\langle H \rangle}_{m+1}$ by assumption, we have $\Breve{s}^\ast \mathbb{R} \Breve{s}_\ast ( G \otimes^\mathbb{L}_{\mathcal{O}_X} E)\in\overline{\langle \breve{s}^\ast H \rangle}_{m+1}$. 
    Consequently, we see that $E$ belongs to $\overline{\langle  G\otimes^{\mathbb{L}} \breve{s}^\ast H \rangle}_{n(m+1)}$ as was to be shown.
\end{proof}

\begin{corollary}\label{cor:separated_flat_diagonal_concrete_cases}
    With notation as in \Cref{prop:pullback_generation_along_flat_separated} and with the shortenings $$\dim_{\oplus}(D_{\operatorname{qc}}(\mathcal{A}_X), D^{-}_{\operatorname{coh}}):=\dim_{\oplus}(D_{\operatorname{qc}}(\mathcal{A}_X), D^{-}_{\operatorname{coh}}(\mathcal{A}_X))\text{, etc.}$$
    The following inequalities hold:
    \begin{align*}
        \dim_{\oplus}(D_{\operatorname{qc}}(\mathcal{A}_X), D^{-}_{\operatorname{coh}}) \leq \dim_\Delta (s,&D^{-}_{\operatorname{coh}}(S)) \cdot (\dim_{\oplus}(D_{\operatorname{qc}}( \mathcal{A}), D^b_{\operatorname{coh}}) + 1) -1,\\
        \dim_{\oplus}(D_{\operatorname{qc}}(\mathcal{A}_X), D^b_{\operatorname{coh}}) \leq \dim_\Delta (s,&D^b_{\operatorname{coh}}(S)) \cdot (\dim_{\oplus}(D_{\operatorname{qc}}(\mathcal{A}), \operatorname{Perf}) + 1) -1,\\
        \dim_{\oplus}(D_{\operatorname{qc}}(\mathcal{A}_X), D^b_{\operatorname{coh}}) \leq \dim_\Delta (s,&\operatorname{Perf}(S)) \cdot (\dim_{\oplus}(D_{\operatorname{qc}}( \mathcal{A}), D^b_{\operatorname{coh}}) + 1) -1,\\
        \dim_{\oplus}(D_{\operatorname{qc}}(\mathcal{A}_X), \operatorname{Perf}) \leq \dim_\Delta(s, &\operatorname{Perf}(S)) \cdot (\dim_{\oplus}(D_{\operatorname{qc}}(\mathcal{A}),\operatorname{Perf}) + 1) -1.
    \end{align*}
\end{corollary}

\begin{example}\label{ex:nc_diagonal_dimension}
    Let $(S,\mathcal{A})$ be a Noether nc.\ scheme and let
    $f \colon X\to S$ be a separated flat morphism of Noetherian schemes. 
    \begin{enumerate}
        \item  Suppose $\dim_\Delta(f,\operatorname{Perf}(X))=1$, e.g.\ when $f$ is \'{e}tale (\Cref{ex:diagdim_etale}). It follows from \Cref{cor:separated_flat_diagonal_concrete_cases} that $\dim_{\oplus}(D_{\operatorname{qc}}(\mathcal{A}_X), D^b_{\operatorname{coh}}(\mathcal{A}_X)) \leq \dim_{\oplus}(D_{\operatorname{qc}}(\mathcal{A}), D^b_{\operatorname{coh}}(\mathcal{A}))$. 

        \item Suppose $ D_{\operatorname{qc}}(\mathcal{A}) =\overline{\langle P \rangle}_{n+1}$ for some perfect complex $P$ and  the integer $m:=\dim_{\Delta}(f,D^b_{\operatorname{coh}}(X))<\infty$.
        Then, again by \Cref{cor:separated_flat_diagonal_concrete_cases}, $\dim D^b_{\operatorname{coh}}(\mathcal{A}_X)\leq m(n+1) - 1<\infty$. A few examples where these assumptions are satisfied include varieties over any field, as well as more generally separated schemes of finite type  $Y$ over a separated Dedekind scheme $X$ where the irreducible components of $Y$ dominate $X$ (e.g.\ proper integral schemes over $\mathbb{Z}$ whose structure morphisms are surjective).
        \end{enumerate}
\end{example}

%%%%%%%%%%%%%%%%%%%%%%%%%%%%%%%%
\section{Descending generation}
\label{sec:descent_conditions}
%%%%%%%%%%%%%%%%%%%%%%%%%%%%%%%%

We prove our results concerning descent and strong generation.
Let us start with a few preliminary lemmas and propositions.

\begin{lemma}\label{lem:bound_above_to_bounded_above_if_resolution_property}
    Let $(X,\mathcal{A})$ be a Noether nc.\ scheme. If $f\colon Y \to X$ is a proper morphism, then $\mathbb{R}f_\ast \mathcal{O}_Y \otimes^{\mathbb{L}} D^{-}_{\operatorname{coh}}(\mathcal{A})$ is contained in $D^{-}_{\operatorname{coh}}(\mathcal{A})$.
\end{lemma}

\begin{proof}
    The tensor action \eqref{eq:tensor-action} restricts to a functor 
    \begin{displaymath}
       D^-_{\operatorname{coh}}(X)\times D^-_{\operatorname{coh}}(\mathcal{A})\to D^-_{\operatorname{coh}}(\mathcal{A})
    \end{displaymath}
    as bounded above flat resolutions exist and coherence can be checked affine locally.
    Hence, the result follows as $\mathbb{R}f_\ast \mathcal{O}_Y\in D^-_{\operatorname{coh}}(X)$ by properness of $f$ (see e.g.\ \cite[Corollary 1.4.12 and Theorem 3.2.1]{EGAIII1:1961}). 
\end{proof}

\begin{lemma}\label{lem:tensoring_perfect}
   Let $(X,\mathcal{A})$ be a Noether nc.\ scheme.
   If $P$ is a perfect complex over $X$, then the endofunctor 
   \[
   (P\otimes^{\mathbb{L}}_{\mathcal{O}_X} -) \colon D_{\operatorname{qc}}(\mathcal{A}) \to D_{\operatorname{qc}}(\mathcal{A})
   \]
   restricts to an endofunctor of $D^b_{\operatorname{coh}}(\mathcal{A})$.
\end{lemma}

\begin{proof}
    We can check the claim locally, in which case we can reduce to the case $P=\mathcal{O}_X$.
\end{proof}

The following is an adaptation of \cite[Proposition 4.4]{Aoki:2021}. It shows that one can obtain generation statements for a triangulated category by studying descendable objects in a tensor triangulated category acting on it.

\begin{proposition}\label{prop:nc_aoki}
    Let $(X,\mathcal{A})$ be a Noether nc.\ scheme. 
    Suppose $f\colon Y \to X$ is an h  cover (so $Y$ is Noetherian) and that either $\mathcal{A}$ is flat or $f$ is additionally flat, i.e.\ an fppf  cover, then there exists an integer $n\geq 0$ such that 
    \begin{displaymath}
        D_{\operatorname{qc}}(\mathcal{A}) = \langle \mathbb{R}f_\ast \mathcal{O}_Y \otimes^{\mathbb{L}}_{\mathcal{O}_X} D_{\operatorname{qc}}(\mathcal{A}) \rangle_n.
    \end{displaymath}
    If $f$ is additionally proper, then one also has
    \begin{displaymath}
        D^{-}_{\operatorname{coh}}(\mathcal{A}) = \langle \mathbb{R}f_\ast \mathcal{O}_Y \otimes^{\mathbb{L}}_{\mathcal{O}_X} D^-_{\operatorname{coh}}(\mathcal{A}) \rangle_n.
    \end{displaymath}
\end{proposition}

\begin{proof}
    Consider the natural morphism $\mathcal{O}_X \to \mathbb{R}f_\ast \mathcal{O}_Y$ and let $C$ denote the object completing the distinguished triangle 
    \[
        C\xrightarrow{\phi} \mathcal{O}_X \to \mathbb{R}f_\ast \mathcal{O}_Y\to C[1].
    \]
    Additionally, let $K_j$ denote the cone of the naturally induced morphism $\phi^{\otimes j} \colon C^{\otimes j }\to \mathcal{O}_X$. 
    As $f$ is an h  cover, $\mathbb{R}f_\ast \mathcal{O}_Y$ is `descendable' in the sense of \cite[Definition 3.18]{Mathew:2016}, see \cite[Proposition 11.25]{Bhatt/Scholze:2017}. (\cite[Definition 3.16]{Balmer:2016} calls these objects nil-faithful.)
    Thus there exists an integer $n\geq 0$ for which $\phi^{\otimes n}=0$ by \cite[Lemma 11.2]{Bhatt/Scholze:2017}. 
   
    Denote $\mathcal{S}:= \{ \mathbb{R}f_\ast \mathcal{O}_Y \otimes^{\mathbb{L}}_{\mathcal{O}_X} E \mid E\in D_{\operatorname{qc}}(\mathcal{A})\}$. 
    We show by induction on $0 \leq i \leq n$  below that $K_i \otimes^{\mathbb{L}}_{\mathcal{O}_X} E\in \langle \mathcal{S}\rangle_i$ for each $E\in D_{\operatorname{qc}}(\mathcal{A})$.
    By our choice of $n$, $K_n$ is isomorphic to $C^{\otimes n}[1]\oplus \mathcal{O}_X$, and so after tensoring on the right with $E$ we see that $E$ is a direct summand of $K_n \otimes^{\mathbb{L}}_{\mathcal{O}_X} E$. 
    Thus $\langle \mathcal{S}\rangle_n = D_{\operatorname{qc}}(\mathcal{A})$.

    It rests to show the induction.
    There is nothing to show for the case $i=0$ ($K_0:=\operatorname{cone}(\operatorname{id})=0)$.
    For the induction step, consider the distinguished triangle obtained from applying the octahedral axiom to the factorization $C^{\otimes i+1}\to C^{\otimes i}\to \mathcal{O}_X$:
    \begin{displaymath}
        \mathbb{R}f_\ast \mathcal{O}_Y \otimes^{\mathbb{L}}_{\mathcal{O}_X} C^{\otimes i} \to K_{i+1}\to K_i \to (\mathbb{R}f_\ast \mathcal{O}_Y\otimes^{\mathbb{L}}_{\mathcal{O}_X} C^{\otimes i}) [1].
    \end{displaymath}
    Tensoring on the right with $E$ then yields the distinguished triangle
    \[
        \mathbb{R}f_\ast \mathcal{O}_Y \otimes^{\mathbb{L}}_{\mathcal{O}_X}  (C^{\otimes i} \otimes^{\mathbb{L}}_{\mathcal{O}_X} E) \to K_{i+1} \otimes^{\mathbb{L}}_{\mathcal{O}_X} E \to K_i \otimes^{\mathbb{L}}_{\mathcal{O}_X} E \to \big( \mathbb{R}f_\ast \mathcal{O}_Y \otimes^{\mathbb{L}}_{\mathcal{O}_X} (C^{\otimes i} \otimes^{\mathbb{L}}_{\mathcal{O}_X} E) \big)[1],
    \]
    from which the induction follows immediately.
    
    The second claim is proven by repeating the argument with $\mathcal{S}= \{ \mathbb{R}f_\ast r\mathcal{O}_Y \otimes^{\mathbb{L}}_{\mathcal{O}_X} E \mid E\in D^-_{\operatorname{coh}}(\mathcal{A}) \}$, and using \Cref{lem:bound_above_to_bounded_above_if_resolution_property}. 
\end{proof} 

The following extends \cite[Theorem 3.5]{Aoki:2021} to nc.\ schemes.

\begin{proposition}\label{prop:nc_coherent_boundedness}
    Let $(X,\mathcal{A})$ be a separated Noether nc.\ scheme and let $f\colon Y \to X$ be a morphism of finite type between separated Noetherian schemes. 
    For any $E\in D^b_{\operatorname{coh}}(\mathcal{A}_Y)$ there exists an integer $n\geq 0$ and an $F\in D^b_{\operatorname{coh}}(\mathcal{A})$ such that $\mathbb{R}\Breve{f}_\ast E\in \overline{\langle F \rangle}_n$, i.e.\ the pushforward is coherently bounded.
\end{proposition}

\begin{proof}
    Suppose first $f$ is an open immersion $j\colon U \to X$.
    Let $G$ be a classical generator for $\operatorname{Perf}(X)$. 
    As the restriction $D^b_{\operatorname{coh}}(\mathcal{A}) \to D^b_{\operatorname{coh}}(\mathcal{A}|_U)$ is a Verdier localization, see \cite[Theorem 4.4]{Elagin/Lunts/Schnurer:2020}, we can find an $E^\prime\in D^b_{\operatorname{coh}}(\mathcal{A})$ with $\Breve{j}^\ast E^\prime=E$. 
    Moreover, by \cite[Theorem 6.2]{Neeman:2021b}, there exists an integer $n$ such that $\mathbb{R}j_\ast \mathcal{O}_U\in \overline{\langle G \rangle}_n$ (in $D_{\operatorname{qc}}(X)$).
    Consequently, $\mathbb{R}j_\ast \mathcal{O}_U \otimes^{\mathbb{L}}_{\mathcal{O}_X} E^\prime$ belongs to $\overline{\langle G \otimes^{\mathbb{L}}_{\mathcal{O}_X} E^\prime\rangle}_n$ (in $D_{\operatorname{qc}}(\mathcal{A})$). 
    However, $\mathbb{R}j_\ast \mathcal{O}_U \otimes^{\mathbb{L}}_{\mathcal{O}_X} E^\prime = \mathbb{R}\Breve{j}_\ast \mathbb{L}\Breve{j}^\ast E^\prime$ by \Cref{prop:strict_form_nc_projection_formula} (as open immersions are flat), and so $\mathbb{R}\Breve{j}_\ast E\in\overline{\langle G \otimes^{\mathbb{L}}_{\mathcal{O}_X} E^\prime\rangle}_n$ and $F:=G \otimes^{\mathbb{L}}_{\mathcal{O}_X} E^\prime\in D^b_{\operatorname{coh}}(\mathcal{A})$ by \Cref{lem:tensoring_perfect}.

    The general case is analogous to \cite[Theorem 3.5]{Aoki:2021}, but we spell it out for convenience of the reader. 
    By Nagata’s compactification theorem, any such morphism $f$ can be factored as an open immersion followed by a proper morphism , see \cite[\href{https://stacks.math.columbia.edu/tag/0F41}{Tag 0F41}]{stacks-project}.
    Hence, the result follows from the first paragraph and the fact that the derived nc.\ pushforward along a proper morphism preserves complexes with bounded and coherent cohomology.
\end{proof}

Our first result is a noncommutative generalization of \cite[Theorem D]{Lank:2024}.
 
\begin{theorem}\label{thm:nc_local_to_global_affine_cover}%(cf. \cite[Theorem D]{Lank:2024})
    Let $(X,\mathcal{A})$ be a separated Noether nc.\ scheme.
    The following are equivalent
    \begin{enumerate}
        \item $D^b_{\operatorname{coh}}(\mathcal{A})$ admits a strong generator,
        \item $D^b_{\operatorname{coh}}(\mathcal{A}|_U)$ admits strong generator for each affine open $U\subseteq X$.
    \end{enumerate}
\end{theorem}

\begin{proof}
    This follows verbatim to the proof of \cite[Theorem D]{Lank:2024} when replacing the necessary arguments with \Cref{prop:strict_form_nc_projection_formula}, \Cref{prop:nc_coherent_boundedness} (only needed for open immersions), and \cite[Corollary 4.4]{DeDeyn/Lank/ManaliRahul:2024}. 
\end{proof}

\begin{corollary}\label{cor:nc_local_to_global_affine_cover_big}
    Let $(X,\mathcal{A})$ be a separated Noether nc.\ scheme.
    The following are equivalent:
    \begin{enumerate}
        \item \label{rmk:nc_local_to_global_affine_cover_big_cond_1} $D_{\operatorname{qc}}(\mathcal{A})$ admits a strong $\oplus$-generator with bounded and coherent cohomology,
        \item \label{rmk:nc_local_to_global_affine_cover_big_cond_2} $D_{\operatorname{qc}}(\mathcal{A}|_U)$ admits a strong $\oplus$-generator with bounded and coherent cohomology for each affine open $U\subseteq X$. 
    \end{enumerate}
\end{corollary}

\begin{proof}
    This follows from the proof of \Cref{thm:nc_local_to_global_affine_cover} (and really that of \cite[Theorem D]{Lank:2024}) which proceeds by constructing strong $\oplus$-generators with bounded and coherent cohomology on the `big' category.
\end{proof}

\begin{example}\label{ex:els_global}
    Let $X$ be a separated scheme of finite type over a perfect field and let $\mathcal{A}$ be a coherent $\mathcal{O}_X$-algebra. 
    Observe that the center\footnote{If one wants, one can view this as a scheme by considering $\underline{\operatorname{Spec}}_X(\operatorname{Z}(\mathcal{A}))$ and we can view $\mathcal{A}$ as living over it.} of $\mathcal{A}$ is also of finite type over our perfect base field. It follows that $D^b_{\operatorname{coh}}(\mathcal{A}|_U)$ admits a strong generator for each affine open $U$ of $X$ by \cite[Remark 2.6]{Elagin/Lunts/Schnurer:2020}. 
    Therefore, by \Cref{thm:nc_local_to_global_affine_cover} there exists a strong generator for $D^b_{\operatorname{coh}}(\mathcal{A})$.
\end{example}

For more exotic topologies, the following theorem allows one to at least check strong $\oplus$-generation locally.

\begin{theorem}\label{thm:nc_local_to_global_h_cover_big}
    Let $(X,\mathcal{A})$ be a separated Noether nc.\ scheme. The following are equivalent
    \begin{enumerate}
        \item\label{item:hcover1} $D_{\operatorname{qc}}(\mathcal{A})$ admits a strong $\oplus$-generator with bounded and coherent cohomology,
        \item\label{item:hcover2} there exists an fppf covering $\{ f_i \colon X_i \to X \}_{i}$ with, for each $i$, $D_{\operatorname{qc}}(\mathcal{A}_{X_i})$ admitting a strong $\oplus$-generator with bounded and coherent cohomology.
    \end{enumerate}
    Additionally, if $\mathcal{A}$ is flat over $X$, then the above are also equivalent to
    \begin{enumerate}[resume]
        \item\label{item:hcover3} the same as \eqref{item:hcover2} but for $\{ f_i \colon X_i \to X \}_{i}$ a (finite) h covering.
    \end{enumerate}
\end{theorem}

\begin{proof} 
    Clearly $\eqref{item:hcover1}$ implies both $\eqref{item:hcover2}$ and $\eqref{item:hcover3}$ by simply taking the identity as a cover.
    So we show that both $\eqref{item:hcover2}$ and $\eqref{item:hcover3}$ imply $\eqref{item:hcover1}$. 
    For this observe that in \eqref{item:hcover2} we may assume $I$ finite and the $X_i$ Noetherian.
    (The $X_i$ are clearly locally Noetherian, so refining and taking an appropriate finite subcovering does the trick.)
    Thus, as fppf coverings are h coverings, it suffices to prove the claim for an h covering  $\{ f_i \colon X_i \to X \}^n_{i=1}$ (where the $X_i$ are automatically Noetherian), and assume that the necessary flatness assumptions are satisfied in order to apply \Cref{prop:strict_form_nc_projection_formula,prop:nc_aoki} below.
    Moreover, by \Cref{cor:nc_local_to_global_affine_cover_big} we can refine the covering and assume that the $X_i$ are affine.

    Let $Y:=\sqcup_i X_i$, then we have an $h$ cover $f\colon Y\to X$ induced from the h covering $\{ f_i \colon X_i \to X \}^n_{i=1}$.
    Moreover, being a finite disjoint union of Noetherian affine schemes, $Y$ is itself Noetherian affine (and hence separated).
    Our hypothesis implies that $D_{\operatorname{qc}}(f^\ast \mathcal{A})$ admits a strong generator with bounded and coherent cohomology (as $D_{\operatorname{qc}}(f^\ast \mathcal{A})=\Pi_i D_{\operatorname{qc}}(f_i^\ast \mathcal{A})$).    
     Choose an integer $N_1\geq 0$ and a $G\in D^b_{\operatorname{coh}}(f^\ast \mathcal{A})$ for which $D_{\operatorname{qc}}(f^\ast \mathcal{A}) = \overline{\langle G \rangle}_{N_1}$. 
     As $f$ is an h cover (and the necessary flatness assumptions hold) there exists an integer $N_2\geq 0$ with $ D_{\operatorname{qc}}(\mathcal{A}) = \overline{\langle \mathbb{R}f_\ast \mathcal{O}_Y \otimes^{\mathbb{L}}_{\mathcal{O}_X} D_{\operatorname{qc}}(\mathcal{A}) \rangle}_{N_2}$ by \Cref{prop:nc_aoki}. 
     However, by \Cref{prop:strict_form_nc_projection_formula} (and the flatness assumptions) $\mathbb{R} f_\ast \mathcal{O}_Y \otimes^{\mathbb{L}} D_{\operatorname{qc}}(\mathcal{A})$ is contained within $\mathbb{R}\Breve{f}_\ast D_{\operatorname{qc}}(f^\ast \mathcal{A})$.
     It follows that $D_{\operatorname{qc}}(\mathcal{A})=\overline{\langle\mathbb{R}\Breve{f}_\ast D_{\operatorname{qc}}(f^\ast \mathcal{A}) \rangle}_{N_2}$, and so, $D_{\operatorname{qc}}(\mathcal{A}) = \overline{\langle \mathbb{R} \Breve{f}_\ast G \rangle}_{N_1 N_2}$. 
     Lastly, it follows from \Cref{prop:nc_coherent_boundedness} that we can find a $G^\prime\in D^b_{\operatorname{coh}}(\mathcal{A})$ and an integer $N_3\geq 0$ such that $\mathbb{R}\Breve{f}_\ast G\in \overline{\langle G^\prime \rangle}_{N_3}$.
     Therefore, $D_{\operatorname{qc}}(\mathcal{A}) = \overline{\langle G^\prime \rangle}_{N_1 N_2 N_3}$ which completes the proof.
\end{proof}

\begin{remark}\label{rem:smooth_descent_compacts}
    In fact, for the smooth topology one can replace `with bounded and coherent cohomology' by `that is compact' in \Cref{thm:nc_local_to_global_h_cover_big}.
    The key observation is that the pushforward along a smooth morphism (between quasi-compact seperated schemes) is \emph{compactly bounded} as opposed to merely \emph{coherently bounded}.
    So, concretely, one can replace $D^b_{\operatorname{coh}}$ by $\operatorname{Perf}$ in \Cref{prop:nc_coherent_boundedness}.
    For a proof of this latter fact we refer to \cite[Proposition 4.9]{DeDeyn/Lank/ManaliRahul:2025}.
    Using this, the proof works verbatim for the compact case.
\end{remark}

Furthermore, exploiting the observation in \cite{Lank/Olander:2024} that the diagonal dimension of \'{e}tale morphisms is one (see \Cref{ex:diagdim_etale}), we can prove a converse for the \'{e}tale topology; showing that admitting a strong $\oplus$-generator with bounded and coherent cohomology is a local question in this topology.

\begin{theorem}\label{thm:etale_cover_strong_oplus_bounded_coherent_exist}
    Let $(X,\mathcal{A})$ be a separated Noether nc.\ scheme and fix an \'{e}tale covering $\{X_i \to X \}_i$.
    The following are equivalent
    \begin{enumerate}
        \item\label{item:etale1} $D_{\operatorname{qc}}(\mathcal{A})$ admits a strong $\oplus$-generator with bounded and coherent cohomology,
        \item\label{item:etale2} $D_{\operatorname{qc}}(\mathcal{A}_{X_i})$ admits a strong $\oplus$-generator with bounded and coherent cohomology for each $i$.
    \end{enumerate}
\end{theorem}

\begin{proof}
    We may assume the \'{e}tale coverings are finite as the target is quasi-compact. 
    Moreover, by possible refining the covering, using \Cref{cor:nc_local_to_global_affine_cover_big}, we may further assume that each $X_i$ is affine and Noetherian. 
    So let $\{ f_i \colon X_i \to X \}_{i=1}^n$ be such an \'{e}tale cover and let $Y:=\sqcup^n_{i=1} X_i$. Each $X_i$ is affine, and so, $Y$ is affine. 
    The natural morphism $Y \to X$ is \'{e}tale, and moreover, separated because its source is such. Hence, $\eqref{item:etale1}\implies\eqref{item:etale2}$ follows from \Cref{ex:nc_diagonal_dimension} and $\eqref{item:etale2}\implies \eqref{item:etale1}$ follows from \Cref{thm:nc_local_to_global_h_cover_big} as an \'{e}tale covering is an fppf covering.
\end{proof}

\begin{remark}
    In the same way as \Cref{rem:smooth_descent_compacts}, one can replace `with bounded and coherent cohomology' by `that is compact' in \Cref{thm:etale_cover_strong_oplus_bounded_coherent_exist}.
\end{remark}

For proper morphisms, one can show statements concerning strong generators for the derived category of bounded and coherent complexes, as opposed to strong $\oplus$-generators.
We show a noncommutative generalization of \cite[Theorem C]{Dey/Lank:2024}, but we start with a small lemma.

\begin{lemma}\label{lem:nc_truncate_down_to_bounded}
    Let $(X,\mathcal{A})$ be a Noether nc.\ scheme, $f \colon Y \to X$ a proper surjective morphism of Noetherian schemes, and $n\geq 0$ an integer. 
    Then $D^b_{\operatorname{coh}}(\mathcal{A})\cap \langle \mathbb{R}\Breve{f}_\ast D^{-}_{\operatorname{coh}}( \mathcal{A}_Y) \rangle_n = \langle \mathbb{R}\Breve{f}_\ast D^b_{\operatorname{coh}}( \mathcal{A}_Y) \rangle_n$.
\end{lemma}

\begin{proof}
    This follows from an analogous argument to the proof of \cite[Lemma 3.10]{Dey/Lank:2024} replacing the necessary arguments with with \Cref{lem:nc_bounded_above} and \Cref{prop:strict_form_nc_projection_formula}.
\end{proof}

\begin{theorem}\label{thm:nc_proper_descent}%(cf. \cite[Theorem C]{Dey/Lank:2024})
    Let $(X,\mathcal{A})$ be a Noether nc.\ scheme. 
    If $f\colon Y \to X$ is a proper surjective morphism of Noetherian schemes,  and either $\mathcal{A}$ is flat over $X$ or $f$ is flat, then there exists an integer $n\geq 0$ such that $D^b_{\operatorname{coh}}(\mathcal{A}) = \langle \mathbb{R}\breve{f}_\ast D^b_{\operatorname{coh}}(\mathcal{A}_Y) \rangle_n$.
\end{theorem}

\begin{proof}
    By \Cref{prop:nc_aoki} there exists an integer $n\geq 0$ such that $D^-_{\operatorname{coh}}(\mathcal{A})=\langle \mathbb{R} f_\ast \mathcal{O}_Y \otimes^{\mathbb{L}} D^{-}_{\operatorname{coh}}(\mathcal{A}) \rangle_n$. 
    However, by \Cref{prop:strict_form_nc_projection_formula} which is applicable by either of the flatness assumptions, $\mathbb{R} f_\ast \mathcal{O}_Y \otimes^{\mathbb{L}} D^{-}_{\operatorname{coh}}(\mathcal{A})$ is contained in $\mathbb{R}\Breve{f}_\ast D^-_{\operatorname{coh}}(\mathcal{A}_Y)$ . 
    Hence, $D^-_{\operatorname{coh}}(\mathcal{A})=\langle\mathbb{R}\Breve{f}_\ast D^-_{\operatorname{coh}}(\mathcal{A}_Y) \rangle_n$ and so the desired claim follows from \Cref{lem:nc_truncate_down_to_bounded}.
\end{proof}

\begin{corollary}\label{cor:rouquier_bounds_from_derived_pushforward}
    With notation as in \Cref{thm:nc_proper_descent}.
    There exists the following inequality of Rouquier dimension: 
    \begin{displaymath}
        \dim D^b_{\operatorname{coh}}(\mathcal{A}) \leq (\dim D^b_{\operatorname{coh}}(\mathcal{A}_Y) + 1) \cdot \min\{n \geq 0 \mid \operatorname{Perf}(\mathcal{A})\subseteq \langle \mathbb{R}\Breve{f}_\ast D^b_{\operatorname{coh}}(\mathcal{A}_Y) \rangle_n \}- 1. 
    \end{displaymath}
\end{corollary}

\begin{proof}
    This follows from the proof of \cite[Proposition 3.15]{Lank/Olander:2024} replacing the necessary ingredients by \Cref{prop:strict_form_nc_projection_formula} and \Cref{thm:nc_proper_descent}.
\end{proof}

Let us give two examples, where this corollary can be used to give (more) explicit bounds on Rouquier dimensions.

\begin{example}\label{ex:bounds_Dbcoh1}%(cf. \cite[Proposition 4.2]{Lank/Venkatesh:2024})
    Let $Y$ be a reduced Nagata scheme of Krull dimension one. Suppose $\mathcal{A}$ is a coherent and flat algebra over $Y$. A similar argument to \cite[Proposition 4.2]{Lank/Venkatesh:2024}, using  \Cref{prop:strict_form_nc_projection_formula} and \Cref{cor:rouquier_bounds_from_derived_pushforward}, gives
    \begin{displaymath}
        \dim  D^b_{\operatorname{coh}}(\mathcal{A}) \leq (\dim  D^b_{\operatorname{coh}}(\mathcal{A}_{Y^\nu})  + 1)(1 + \underset{p\in \operatorname{Sing}(Y)}{\max} \{ \delta_p \}) -1,
    \end{displaymath}
    where $Y^\nu \to Y$ is the normalization of $Y$ and $\delta_p$ is the $\delta$-invariant of $X$ at $p$ (see \cite[\href{https://stacks.math.columbia.edu/tag/0C1T}{Tag 0C1T}]{stacks-project}).
\end{example}

\begin{example}\label{ex:bounds_Dbcoh2}
    Fix an integer $n\geq 2$ and let $i\colon X \to \mathbb{P}^n_{\mathsf{k}}$ be a smooth hypersurface of degree $d$, where $\mathsf{k}$ is a field of characteristic zero. 
    Let $\mathcal{L} := i^\ast \mathcal{O}_{\mathbb{P}^n_{\mathsf{k}}}(1)$, and put $C:=\operatorname{Spec}(\oplus_{k\geq 0} H^0(X,\mathcal{L}^{\otimes k}))$ the affine cone over $X$ (w.r.t.\ $\mathcal{L}$).

    Then, $C$ has an isolated singularity at the origin $o$, which can be resolved by blowing up (this point).
    In addition, this blow-up can be identified with the vector bundle $\mathbf{V}(\mathcal{L}):=\underline{\operatorname{Spec}}_X(\operatorname{Sym}(\mathcal{L}))$ over $X$ such that the exceptional divisor $E$ of the blow-up is scheme-theoretically identified with the zero section of this vector bundle, and so in particular is isomorphic to $X$. 
    (See e.g.\ \cite[\S 7]{VandenBergh:2004}.)
    To summarise, we obtain the following diagram
    \begin{displaymath}
        \begin{tikzcd}[ampersand replacement=\&,column sep = 3em]
            {E\cong X} \&  {\operatorname{Bl}_o C \cong \mathbf{V}(\mathcal{L}) } \\
            \{o\} \& C\rlap{ .}
            \arrow[hook, from=1-1, to=1-2]
        	\arrow[from=1-1, to=2-1]
        	\arrow[from=1-2 , to=2-2]
        	\arrow[hook, from=2-1, to=2-2]
        \end{tikzcd}
    \end{displaymath}
    Let $\mathcal{A}$ be a flat and coherent algebra over $C$. Suppose $d \geq n$. A similar argument to \cite[Corollary 4.7]{Lank/Venkatesh:2024} gives
    \begin{displaymath}
    \dim D^b_{\operatorname{coh}}(\mathcal{A}) \leq (\dim D^b_{\operatorname{coh}}(\mathcal{A}_{\operatorname{Bl}_o C}) + 1)(1+ 2(d-n)) -1.
    \end{displaymath}
    (Again, substituting the necessary ingredients by \Cref{prop:strict_form_nc_projection_formula} and \Cref{cor:rouquier_bounds_from_derived_pushforward}).
\end{example}

Next, combining \Cref{thm:etale_cover_strong_oplus_bounded_coherent_exist,thm:nc_proper_descent} we obtain equality of Rouquier dimensions for the bounded derived category of coherent sheaves along proper(=finite) \'{e}tale morphisms.
This yields a version, valid for Noether nc.\ schemes, of \cite[Theorem 2]{Sosna:2014} (see \Cref{cor:rouq_dim_surjective_finite_etale} below), whose proof, in particular, does not require any explicit use of enhancements.
We start with some lemmas.

\begin{lemma}\label{lem:section_retract_adjunction}
    Let $L \colon \mathcal{C} \rightleftarrows \mathcal{D} \colon R$ be an adjoint pair of functors between (any) categories. If $c$ is a retract of $R (d)$ in $\mathcal{C}$, then the unit map $\eta_c \colon c \to R L (c)$ is a split monomorphism in $\mathcal{C}$.
    A similar (mirrored) statement can be made for the counit being a split epimorphism.
\end{lemma}

\begin{proof}
    Indeed, let $f\colon c \to R (d)$ be a morphism with retraction $g \colon R(d) \to c$, i.e.\ $\operatorname{id}_{c} = g\circ f$. 
    We know that $f$ factors through $\eta_c$ by adjunction, say $f = F \circ \eta_c$. 
    Then, $\operatorname{id}_c = (g \circ F) \circ \eta_c$, which shows the desired claim. 
\end{proof}

\begin{lemma}\label{lem:compact_gen_pushforward_finite_faithfully_flat}
    Let $f\colon Y \to X$ be a finite faithfully flat morphism between Noetherian schemes. 
    Then $\mathbb{R}f_\ast=f_\ast \colon \operatorname{Perf}(Y)\to \operatorname{Perf}(X)$ preserves classical generators.
\end{lemma}

\begin{proof}
    The main point is that, as $f$ is a finite faithfully flat morphism, $f_\ast \mathcal{O}_Y$ is a finite locally free $\mathcal{O}_X$-module with full support.
    Indeed, as a finite morphism is affine, one can check the claim with source and target being affine schemes. But, if $g\colon \operatorname{Spec}(S) \to \operatorname{Spec}(R)$ is a finite faithfully flat morphism of Noetherian affine schemes, $g_\ast S=S|_R$ is a projective $R$-module of finite positive rank.
    (Observe that this also shows that $f_\ast$ preserves perfect complexes, as this can also be checked affine locally.)
    % Faithfully flat means surjective on underlying topological spaces, so full support.
    
    To show the claim,  let $G$ be a classical generator for $\operatorname{Perf}(Y)$ and let $Q$ be an arbitrary object in $\operatorname{Perf}(X)$. 
    By the previous paragraph, for all $p\in X$, $(f_\ast \mathcal{O}_Y)_p= \mathcal{O}_{X,p}^{\oplus n_p}$ for some $n_p>0$ and so finitely builds $Q_p$. 
    Consequently, by \cite[Theorem 1.7]{BILMP:2023} $Q$ is finitely built by $f_\ast \mathcal{O}_Y \otimes_{\mathcal{O}_X} P $ for some classical generator $P$ of $\operatorname{Perf}(X)$. 
    The projection formula tells us that $f_\ast \mathcal{O}_Y \otimes_{\mathcal{O}_X} P=f_\ast f^\ast P$.
    Clearly, $f^\ast P$ is finitely built by $G$ (as the latter is a classical generator), and so, $f_\ast f^\ast P$ is finitely built by $f_\ast G$. Hence, $Q$ is finitely built by $f_\ast G$, which was to be shown.
\end{proof}

\begin{lemma}\label{lem:counit_split_faithfully_flat_of_finite_type}
    Let $f\colon Y \to X$ be a finite faithfully flat morphism between Noetherian schemes.
    Then the unit map $E\to \mathbb{R}f_\ast \mathbb{L}f^\ast E = f_\ast f^\ast E$ splits for all $E\in D_{\operatorname{qc}}(X)$. 
\end{lemma}

\begin{proof}
    We first show the case when $E$ is perfect.
    For this, define $\mathcal{T}$ to be the full subcategory of $\operatorname{Perf}(X)$ consisting of those objects which are direct summands of $f_\ast Q$ for some perfect complex $Q$ over $Y$. 
    The proof of \cite[Lemma A.1]{Hall:2016} shows that $\mathcal{T}$ is a thick subcategory.
    Consequently, $\mathcal{T}=\operatorname{Perf}(X)$ as the former contains a classical generator by \Cref{lem:compact_gen_pushforward_finite_faithfully_flat}.
    Thus, for every perfect complex $P$ over $X$, there exists a perfect complex $Q$ over $Y$ with $P$ a direct summand of $f_\ast Q$. 
    It follows from \Cref{lem:section_retract_adjunction} that the unit map $P\to f_\ast f^\ast P$ splits.

    Next, we show the claim for an arbitrary $E$ in $D_{\operatorname{qc}}(X)$. 
    The natural map $\mathcal{O}_X \to f_\ast \mathcal{O}_Y$ splits by the previous paragraph.
    Hence, by tensoring with $E$, we get the required splitting $E \to f_\ast f^\ast E$.
\end{proof}

\begin{proposition}\label{prop:faithfully_flat_proper_rouq_dim}
    Let $(X,\mathcal{A})$ be a Noether nc.\ scheme. If $f\colon Y \to X$ is a finite faithfully flat morphism, then $\dim D^b_{\operatorname{coh}}(\mathcal{A}) \leq \dim D^b_{\operatorname{coh}}(\mathcal{A}_Y)$.
\end{proposition}

\begin{proof}
    By \Cref{lem:counit_split_faithfully_flat_of_finite_type}  the natural morphisn $\mathcal{O}_X\to f_\ast \mathcal{O}_Y$ splits.
    Consequently, $C=0$ in \Cref{prop:nc_aoki} and we can take $n=1$ in the statement.
    Therefore, we can also take $n=1$ in \Cref{thm:nc_proper_descent} yielding the claim.   
\end{proof}

\begin{corollary}\label{cor:rouq_dim_surjective_finite_etale}
    Let $(X,\mathcal{A})$ be a Noether nc.\ scheme. If $f\colon Y \to X$ is a finite \'{e}tale cover, then 
    \begin{align}
        \label{eq:rouq_dim_surjective_finite_etale1}    \dim  D^b_{\operatorname{coh}}(\mathcal{A}_Y) &= \dim D^b_{\operatorname{coh}}(\mathcal{A})\quad\text{and} \\ 
        \label{eq:rouq_dim_surjective_finite_etale2}    \dim_\oplus(D_{\operatorname{qc}}(\mathcal{A}_Y),D^b_{\operatorname{coh}}(\mathcal{A}_Y)) &= \dim_\oplus(D_{\operatorname{qc}}(\mathcal{A}), D^b_{\operatorname{coh}}(\mathcal{A})).
    \end{align}
    In particular, this is applicable to  $\operatorname{Spec}(\mathsf{L})\times_\mathsf{k} X\to X$ for any finite Galois extension $\mathsf{L}/\mathsf{k}$.
\end{corollary}

\begin{proof}
    Let us first show that both $\breve{f}^\ast D^b_{\operatorname{coh}}(\mathcal{A})\subseteq D^b_{\operatorname{coh}}( \mathcal{A}_Y)$ ($\breve{f}^\ast D^b_{\operatorname{coh}}(\mathcal{A})$ is contained in $D^b_{\operatorname{coh}}(\mathcal{A}_Y)$ as $f$ is flat) and $\breve{f}^\ast D_{\operatorname{qc}}(\mathcal{A})\subseteq D_{\operatorname{qc}}(\mathcal{A}_Y)$ are essentially dense. 
    Indeed, it follows from \Cref{ex:diagdim_etale} that $E$ is a direct summand of $\breve{f}^\ast \breve{f}_\ast E$ for all $E\in D_{\operatorname{qc}}(\mathcal{A}_Y)$. 
    (As $E$ is isomorphic to $\breve{p}_{2,\ast} (\Delta_\ast \mathcal{O}_Y \otimes^{\mathbb{L}} \breve{p}_1^\ast E)$ by \Cref{prop:strict_form_nc_projection_formula} applied to $\Delta$, and $\breve{f}^\ast \breve{f}_\ast E$ is isomorphic to $\breve{p}_{2,\ast} \breve{p}_1^\ast E$ by \Cref{prop:nc_base_change}, where the $p_i$ are the projections).
        
    The first paragraph shows one inequality of \eqref{eq:rouq_dim_surjective_finite_etale1}, whilst the other follows from \Cref{prop:faithfully_flat_proper_rouq_dim}.
    Similarly, one inequality of \eqref{eq:rouq_dim_surjective_finite_etale2} follows from the first paragraph.
    For the other, as the natural morphism $\mathcal{O}_X \to f_\ast \mathcal{O}_Y$ splits by \Cref{lem:counit_split_faithfully_flat_of_finite_type} tensoring with an object $E \in D_{\operatorname{qc}}(\mathcal{A})$ and using \Cref{prop:strict_form_nc_projection_formula} yields a splitting $E\to \breve{f}_\ast \breve{f}^\ast E$. 
    Thus, since $\breve{f}_\ast D^b_{\operatorname{coh}}(\mathcal{A}_Y)$ is contained in $D^b_{\operatorname{coh}}(\mathcal{A})$ as $f$ is finite, the inequality follows.
\end{proof}

\begin{example}
    Let $X$ be a variety over a (perfect) field $\mathsf{k}$. Suppose there is a finite Galois extension $\mathsf{L}/\mathsf{k}$ such that $X_\mathsf{L}:=\operatorname{Spec}(\mathsf{L})\times_\mathsf{k} X$ is a toric variety (e.g.\ $X$ is a Severi--Brauer variety). 
    These are sometimes called `arithmetic toric varieties', see e.g.\ \cite{ Elizondo/Lima-Filho/Sottile/Teitler:2014} (see also \cite{Voskresenskij/Klyachko:1985, Merkurjev/Panin:1997,Duncan:2016}).
    Consider coherent algebra $\mathcal{A}$ over $X$ of full support. 
    We claim that $\dim D^b_{\operatorname{coh}}(\mathcal{A})=\dim X$. 
    The inequality $\dim X\leq \dim D^b_{\operatorname{coh}}(\mathcal{A})$ holds generally by \cite[Theorem C]{Bhaduri/Dey/Lank:2023}. 
    For the other inequality, it suffices to prove the claim for the nc.\ scheme $(X_\mathsf{L},\mathcal{A}_{X_\mathsf{L}})$ by \Cref{cor:rouq_dim_surjective_finite_etale}, i.e.\ we may assume $X$ is toric. 
    Next pick a toric resolution of singularities $f\colon Y \to X$ (i.e.\ $f$ is a resolution of singularities and $Y$ is a smooth toric variety over $\mathsf{k}$),  see e.g.\ \cite[\S 8]{Danilov:1978} for the existence of such.
    As toric varieties have rational singularities (see e.g.\ \cite[Proposition 8.5.1]{Danilov:1978}), we have that the natural map $\mathcal{O}_X\to \mathbb{R}f_\ast \mathcal{O}_{Y}$ is an isomorphism.
    \Cref{cor:rouquier_bounds_from_derived_pushforward} yields $\dim D^b_{\operatorname{coh}}(\mathcal{A}_{Y}) \geq \dim D^b_{\operatorname{coh}}(\mathcal{A})$, so we may replace $X$ by $Y$ and thereby assume that the toric variety is smooth. 
    However, as the diagonal dimension of a smooth toric variety coincides with its Krull dimension \cite{Favero/Huang:2023, Hanlon/Hicks/Lazarev:2023, Brown/Erman:2024}\footnote{This result was proven around the same time, independently, by three different groups, and resolved a conjecture of Bondal.}, the desired claim follows by \Cref{cor:separated_flat_diagonal_concrete_cases}.
\end{example}

The last result in this section is a global version of \cite[Theorem 3.6]{Letz:2021} and a noncommutative variation of \cite[Theorem 1.7]{BILMP:2023}.
The proof is a straighforward addaptation of loc.\ cit.\ and relies on the global-to-local principle from \cite{Stevenson:2013}.
In the separated setting, however, it is possible to give an inductive proof (over an affine open cover) to reduce to affine situation, similar to how one shows Zariski descent.

\begin{theorem}\label{thm:local_global_stalk_generation} 
    Let $(X,\mathcal{A})$ be a Noether nc.\ scheme.
    Fix $E$, $G\in D^b_{\operatorname{coh}}(\mathcal{A})$.
    Suppose $E_x\in\langle G_x \rangle\subseteq D^b_{\operatorname{coh}}(\mathcal{A}_x)$ for all $x\in X$, then there exists a perfect complex $Q$ over $X$ such that $E\in \langle Q \otimes^{\mathbb{L}}_{\mathcal{O}_X} G \rangle\subseteq D^b_{\operatorname{coh}}(\mathcal{A})$. 
    In particular, we may choose $Q\in\operatorname{Perf}(X)$ any classical generator (independent of $E$).
\end{theorem}

\begin{proof}
\label{rmk:local-to-global-argument}
    The argument is the same as that of \cite{BILMP:2023}, so we simply sketch it for convenience of the reader (but refer to loc.\ cit.\ for some notation).

    Firstly, note that $\langle Q \otimes^{\mathbb{L}}_{\mathcal{O}_X} G \rangle$ is exactly the smallest thick $\operatorname{Perf}(X)$-submodule of $D^b_{\operatorname{coh}}(\mathcal{A})$ containing $G$.

    Secondly, observe that $D^b_{\operatorname{coh}}(\mathcal{A})=K(\operatorname{Inj}(\mathcal{A}))^c$ \cite[Proposition 2.3]{Krause:2005}, where $K(\operatorname{Inj}(\mathcal{A}))$ denotes the homotopy  category of quasi-coherent injective $\mathcal{A}$-modules.
    
    Thirdly, it is possible to `complete' the action of $\operatorname{Perf}(X)$ on $D^b_{\operatorname{coh}}(\mathcal{A})$ to an action $\odot$ of $D_{\operatorname{qc}}(X)$ on $K(\operatorname{Inj}(\mathcal{A})$, see \cite[\S3]{Stevenson:2014}.

    The claim, then follows from the local-to-global principle for the latter action.  
    Indeed, if $E_x\in\langle G_x \rangle$ for all $x\in X$, then 
    \begin{displaymath}
        E\in \operatorname{Loc}^\odot(E)=\operatorname{Loc}^\odot\{\Gamma_x E  \mid x\in X\}\subseteq \operatorname{Loc}^\odot\{\Gamma_x G  \mid x\in X\}=\operatorname{Loc}^\odot(G),
    \end{displaymath}
    where the equalities come from the local-to-global principle \cite[Definition 6.1 \& Theorem 6.9]{Stevenson:2013}.
    Thus $E$ is in the smallest $\operatorname{Perf}(X)$-module of $D^b_{\operatorname{coh}}(\mathcal{A})$ containing $G$, i.e.\ in $\langle Q \otimes^{\mathbb{L}}_{\mathcal{O}_X} G \rangle$.
\end{proof}

\begin{remark}
    The appearance of the perfect complex $Q$ in  \Cref{thm:local_global_stalk_generation} is subtle, necessary, and not every object with full support would be an appropriate choice. 
    In the affine setting, any perfect complex with full support is a classical generator for the category of perfect complexes \cite{Neeman:1992}, and moreover one can simply take $Q=\mathcal{O}_X$ in \Cref{thm:local_global_stalk_generation}.
    However, this can horribly fail in the non-affine setting; e.g.\ $\mathcal{O}_{\mathbb{P}^1_{\mathsf{k}}}$ is not a classical generator for $D^b_{\operatorname{coh}}(\mathbb{P}^1_{\mathsf{k}})$ (with $\mathsf{k}$ some field), so the addition of the $Q$ is really necessary.
\end{remark}

%%%%%%%%%%%%%%%%%%%%%%%%%%%%%%%%
\section{Applications to Azumaya algebras}
\label{sec:application_azumaya}
%%%%%%%%%%%%%%%%%%%%%%%%%%%%%%%%

We now bring our focus to Azumaya algebras.
Their properties when it comes to generation are tightly intertwined with those of their center. 

\begin{theorem}\label{thm:azumaya_strong_generator_iff_center}
    Let $(X,\mathcal{A})$ be a separated Azumaya scheme. 
    Then $D_{\operatorname{qc}}(\mathcal{A})$ has finite $\oplus$-dimension with respect to the bounded coherent complexes if and only if $D_{\operatorname{qc}}(X)$ has.
    Moreover, assuming $\mathcal{A}$ can be split by a finite \'{e}tale morphism, we have $\dim D^b_{\operatorname{coh}}(\mathcal{A})=\dim D^b_{\operatorname{coh}}(X)$.
\end{theorem}

\begin{proof}
    Assume $\mathcal{A}$ is split, i.e.\ the endomorphism sheaf of a finite locally free module (or even a finite free module if one prefers).
    Then, by Morita theory, $\operatorname{Qcoh}(\mathcal{A})\cong\operatorname{Qcoh}(X)$ (and this restricts to coherent modules).
    Consequently, noting that by definition every Azumaya algebra is \'{e}tale locally split, using \Cref{thm:etale_cover_strong_oplus_bounded_coherent_exist} and \Cref{cor:rouq_dim_surjective_finite_etale} we can reduce to the split setting, showing both claims.
\end{proof}

This has some pleasant consequences.
The following is \cite[Main Theorem]{Aoki:2021} for Azumaya algebras.

\begin{corollary}\label{cor:nc_aoki}
    Let $(X,\mathcal{A})$ be an Azumaya scheme whose underlying scheme $X$ is a quasi-compact separated quasi-excellent scheme of finite Krull dimension.
    Then $D^b_{\operatorname{coh}}(\mathcal{A})$ admits a strong generator.
\end{corollary}

\begin{proof}
    Note that the proof of \cite[Main Theorem]{Aoki:2021} proceeds by constructing a strong $\oplus$-generator for $D_{\operatorname{qc}}(X)$ with bounded and coherent cohomology, which immediately gives us the required result by \Cref{thm:azumaya_strong_generator_iff_center}.
\end{proof}

Our next result shows how to identify strong generators for Azumaya algebras by lifting one from its center.
In particular this allows us to apply \cite[Theorem A]{BILMP:2023}, which allows us to explicitly identify strong generators for nc.\ schemes of prime characteristic.
We start with a lemma.

\begin{lemma}\label{lem:generate_by_perfect_pushforwards}
    Let $(X,\mathcal{A})$ be an Azumaya scheme. Suppose $Z$ is a closed subset of $X$. If $P$ is an object of $\operatorname{Perf}_Z (\mathcal{A})$ for which $\operatorname{supp}(P)=Z$, then $D^b_{\operatorname{coh},Z}(\mathcal{A})$  coincides with the thick subcategory generated by $P$ and objects of the form $\Breve{i}_\ast \mathbb{L}\Breve{i}^\ast P$, where $i \colon Y \to X$ is a closed immersion from an integral scheme $Y$ contained in $Z$.
\end{lemma}

\begin{proof}
    This follows verbatim to \cite[Theorem A]{Dey/Lank:2024}, as the required lemmas in ibid.\ also have the appropriate analogues in this setting---it is worthwhile noting that the `Azumaya' assumption is needed as one needs the pullback of $\mathcal{A}$ under closed immersions to be generically semisimple for the proofs of ibid.\ to go through.
\end{proof}

\begin{proposition}\label{prop:lifting_strong_gen_from_centre_for_azumaya}
    Let $(X,\mathcal{A})$ be an Azumaya scheme and $Z$ be a closed subset of $X$.
    Suppose $G$ is a classical generator for $D^b_{\operatorname{coh},Z}(X)$, then $\pi^\ast G=G\otimes_{\mathcal{O}_X} \mathcal{A}$ is a classical generator for $D^b_{\operatorname{coh},Z}(\mathcal{A})$.
\end{proposition}

\begin{proof}
    As usual, let $\pi\colon(X,\mathcal{A})\to X$ denote the structure morphism.
    Applying \Cref{lem:generate_by_perfect_pushforwards} with $P=\mathcal{A}$, we see that $D^b_{\operatorname{coh},Z}(\mathcal{A})$ is generated by 
    \begin{displaymath}
        \Breve{i}_\ast \mathbb{L}\Breve{i}^\ast \mathcal{A} \cong i_\ast \mathcal{O}_Y \otimes \mathcal{A} = \pi^\ast \mathcal{O}_Y
    \end{displaymath}
    for $i\colon Y\to X$ closed integral subschemes contained in $Z$.
    However, applying the lemma again, this time to the Azumaya scheme $(X,\mathcal{O}_X)$ we see that
    \begin{displaymath}
    D^b_{\operatorname{coh},Z}(X)= \langle i_\ast \mathcal{O}_Y\mid \text{$i\colon Y\to X$ closed integral subscheme contained in }Z \rangle.
    \end{displaymath}
    By assumption this also equals $\langle G\rangle$.
    Hence, it follows that
    \begin{displaymath}
        \langle \pi^\ast G\rangle=D^b_{\operatorname{coh}}(\mathcal{A})
    \end{displaymath}
    which shows our claim.
\end{proof}

\begin{remark}
    Note that a `strong generation' version of \Cref{prop:lifting_strong_gen_from_centre_for_azumaya} (for $\varnothing\neq Z\neq X$) is possibly vacuous.
    Indeed, for an affine Noetherian integral scheme $X$, it was shown by Elagin--Lunts \cite[Theorem 3.1]{Elagin/Lunts:2018} that any thick subcategory of $D^b_{\operatorname{coh},Z}(X)$ admitting a strong generator is either the trivial subcategory consisting of only zero objects or coincides with $D^b_{\operatorname{coh}}(X)$.   
\end{remark}

\begin{example}\label{ex:nc_Frobenius}%(cf. \cite[Theorem A]{BILMP:2023})  
    Let $(X,\mathcal{A})$ be a Azumaya scheme where $X$ is $F$-finite of characteristic $p$.
    Here \textbf{$F$-finite} means that the $i$-th iterate of the Frobenius morphism $F^i \colon X \to X$ is a finite morphism for some integer $i$ (equivalently, for all $i$).    
    Suppose $Q$ is a classical generator for $\operatorname{Perf}(X)$ and $e$ is large enough (this can be made precise, see \cite[Theorem A]{BILMP:2023}), then $F_\ast^e Q \otimes_{\mathcal{O}_X} \mathcal{A}$ is a classical generator for $D^b_{\operatorname{coh}}(\mathcal{A})$.
    Indeed, it follows form loc.\ cit.\ that $F_\ast^e Q$ is a classical generator for $D^b_{\operatorname{coh}}(X)$, so this follows immediately from \Cref{prop:lifting_strong_gen_from_centre_for_azumaya}. Additionally, if $X$ is separated, then $F_\ast^e Q \otimes_{\mathcal{O}_X} \mathcal{A}$ is a strong generator for $D^b_{\operatorname{coh}}(\mathcal{A})$ as the latter admits a strong generator by \Cref{cor:nc_aoki}---an $F$-finite scheme is quasi-excellent of finite Krull dimension, see e.g.\ \cite[Proposition 1.1 \& Theorem 2.5]{Kunz:1976}.
\end{example}

As some context for our last result, recall that a conjecture by Orlov \cite{Orlov:2009} states that the Rouquier dimension of the bounded derived category of coherent sheaves on a smooth variety $X$ equals the Krull dimension of $X$.
A weaker version of this problem, replacing Rouquier dimension by  \textit{countable Rouquier dimension}, has been solved by Olander in \cite{Olander:2023}. 
The countable Rouquier dimension is defined exactly as the Rouquier dimension, but allowing a countable number of objects to generate instead of looking at a single generator.
We finish this section by extending \cite[Theorem 4]{Olander:2023} to Azumaya algebras over derived splinters.
Recall that a Noetherian scheme $X$ is said to be a \textbf{derived splinter} if for all proper surjective morphisms $f \colon Y \to X$ the natural morphism $\mathcal{O}_X \to \mathbb{R}f_\ast \mathcal{O}_Y$ splits. 
In characteristic zero being a derived splinter is equivalent to having {rational singularities}; see \cite{Kovacs:2000, Bhatt:2012} for details.

\begin{proposition}\label{prop:countable_rouq_dim_derived_splinter}
    Let $(X,\mathcal{A})$ be a separated Azumaya scheme and assume that $X$ admits a resolution of singularities and has finite Krull dimension. 
    If $X$ is a derived splinter, then the countable Rouquier dimension of $D^b_{\operatorname{coh}}(\mathcal{A})$ is at most the Krull dimension of $X$.
\end{proposition}

\begin{proof}
    Let $f\colon Y \to X$ be a resolution of singularities. Then $Y$ is separated as $f$ is proper and $X$ is separated, in particular, $Y$ has affine diagonal. 
    Moreover, $\mathcal{A}_{Y}$ is an Azumaya algebra over a regular separated Noetherian scheme of finite Krull dimension, and so the global dimension of each stalk of $\mathcal{A}_{Y}$ is at most $\dim X$ (see e.g.\ \cite[Theorem 1.8]{Auslander/Goldman:1961}). 
    Hence, it follows from the proof of \cite[Theorem 4]{Olander:2023} that $D^b_{\operatorname{coh}}(\mathcal{A}_{Y})$ has countable Rouquier dimension at most $\dim X$, see also \cite[Remark 3.19]{DeDeyn/Lank/ManaliRahul:2024}. 
    Let $E$ be an arbitrary object in $D^b_{\operatorname{coh}}(\mathcal{A})$.
    By \Cref{prop:strict_form_nc_projection_formula} we have $\mathbb{R}\Breve{f}_\ast \mathbb{L} \Breve{f}^\ast E\cong \mathbb{R}f_\ast \mathcal{O}_{Y} \otimes^{\mathbb{L}}_{\mathcal{O}_X} E$. 
    As $X$ is a derived splinter, the natural map $\mathcal{O}_X \to \mathbb{R}f_\ast \mathcal{O}_{Y}$ splits in $D^b_{\operatorname{coh}}(X)$, and so $E$ is a direct summand of $\mathbb{R}\Breve{f}_\ast \mathbb{L} \Breve{f}^\ast E$. 
    Using \Cref{lem:nc_bounded_above} one sees that $\mathbb{R}\Breve{f}_\ast \mathbb{L} \Breve{f}^\ast E$ is an object of $D^{-}_{\operatorname{coh}}(\mathcal{A})$.
    Consequently, $E$ belongs to $\langle \mathbb{R}\Breve{f}_\ast D^b_{\operatorname{coh}}(\mathcal{A}_{Y}) \rangle_1$ by \Cref{lem:nc_truncate_down_to_bounded}. 
    Thus the countable Rouquier dimension of $D^b_{\operatorname{coh}}(\mathcal{A})$ is at most that of $D^b_{\operatorname{coh}}(\mathcal{A}_{Y})$, which completes proof.
\end{proof}

%%%%%%%%%%%%%%%%%%%%%%%%%%%%%%%%
\section{Big vs.\ small generation---the elephant in the room}
\label{sec:big_vs_small_elephant}
%%%%%%%%%%%%%%%%%%%%%%%%%%%%%%%%

In this final section, we discuss what we know---and do \emph{not} know---concerning the relation between `big and small generation'.
There are two forms of generation that arose above.
One can consider the existence of a strong generator of some chosen essentially small subcategory\footnote{Note that a triangulated category that is not essentially small, e.g.\ it is non-trivial with arbitrary small coproducts, can never admit a strong generator.}(=small generator) or the existence of a strong $\oplus$-generator with respect to some (essentially small) subcategory(=big generator). 
More precisely, fix a Noether nc.\ scheme $(X,\mathcal{A})$ (or more generally $X$ qcqs with $\mathcal{A}$ coherent as $\mathcal{O}_X$-module); above we looked at finiteness of either
\begin{displaymath}
    \dim \left(D^b_{\operatorname{coh}}(\mathcal{A})\right)\text{, i.e.\ the small}
\end{displaymath}
or of
\begin{displaymath}
    \dim_{\oplus}\left(D_{\operatorname{qc}}(\mathcal{A}), D^b_{\operatorname{coh}}(\mathcal{A})\right)\text{, i.e.\ the big}.
\end{displaymath}
Much of the literature focuses on small strong generation, as the corresponding categories are classically more studied/of more interest.
Yet, the existence of small strong generators is, especially recently, shown using `big techniques'.
In addition, our work above shows that big generators are relatively well-behaved with respect to descent for various topologies, whilst it is not immediately clear how to adapt the proofs to show the corresponding statements for small generators.
As big generators generally yield small generators, it is only naturally naive to hope that the two notions are the same.

Well let us just start with the bad news.
In general, one cannot expect the notions to be the same, as the following example shows.
\begin{example}[\cite{Stevenson:2025}]\label{ex:gregs}
    There exist commutative coherent rings $R$ of infinite global dimension where
    \begin{displaymath}
        D^b_{\operatorname{mod}}(\operatorname{Mod}(R))=D^b(\operatorname{mod}(R))=\operatorname{Perf}(R)
    \end{displaymath}
    admits a strong generator but for which $D(\operatorname{Mod}(R))$ does not admit a strong $\oplus$-generator with bounded finitely presented cohomology.
    Indeed, the existence of the latter would imply that the global dimension of $R$ is finite by \cite[Theorem A]{DeDeyn/Lank/ManaliRahul:2024}, but this contradicts the fact that $R$ has infinite global dimension.
\end{example}

However, there are settings where big and small generators are known to coincidence generally. 
Indeed, our work \cite[Theorem A and B]{DeDeyn/Lank/ManaliRahul:2024}, shows that, as long as $\mathcal{A}$ is Noetherian, a compact object of $D_{\operatorname{qc}}(\mathcal{A})$ is a strong $\oplus$-generator if and only if it is a strong generator of the subcategory of compact objects. 
So, in a \emph{Noether} (and `regular' i.e.\ when $D^b_{\operatorname{coh}}=\operatorname{Perf}$) setting big and small coincide for bounded coherent objects.
This raises the following natural question.
\begin{question}\label{question:nc_setup}
    Let $(X,\mathcal{A})$ be a Noether nc.\ scheme and fix some $G\in D^b_{\operatorname{coh}}(\mathcal{A})$.
    Is $G$ a strong generator for $D^b_{\operatorname{coh}}(\mathcal{A})$ if and only if it is a strong $\oplus$-generator for $D_{\operatorname{qc}}(\mathcal{A})$?
\end{question}
One implication (the `if') is always true and holds very generally (see e.g.\ \cite[Lemma 2.14]{DeDeyn/Lank/ManaliRahul:2024}).
The `only if' direction also holds, granted we already know the existence of some strong $\oplus$-generator with bounded and coherent cohomology.
However, it is unclear how to prove this direction without using the existence of this strong $\oplus$-generator.

A nice setting, under which our noncommutative set-up falls, for more generally understanding the connection between small and big generators is that of `approximable triangulated categories' as introduced by Neeman in \cite{Neeman:2021a}. Let us briefly recall some definitions, the following first appeared in \cite[Definition 0.21]{Neeman:2021b}.

\begin{definition}
    A triangulated category $\mathcal{T}$ with small coproducts is said to be \textbf{approximable} if there exists a compact generator $G$ for $\mathcal{T}$, a $t$-structure $\tau=(\mathcal{T}^{\leq 0}, \mathcal{T}^{\geq 0})$ on $\mathcal{T}$ and a positive integer $A$ satisfying
    \begin{enumerate}
        \item $G\in\mathcal{T}^{\leq A}$, $\operatorname{Hom}(G,\mathcal{T}^{\leq -A})=0$ and
        \item for each object $F$ in $\mathcal{T}^{\leq 0}$ there exists a distinguished triangle
        \begin{displaymath}
            E \to F \to D \to E[1]
        \end{displaymath}
        with $D\in\mathcal{T}^{\leq -1}$ and $E$ can be obtained from $\{G[i] : -A \leq i \leq A\}$ using only $A$ cones, direct summands, and small coproducts.
    \end{enumerate}
    % A triple $(G,\tau,A)$ as above is called \textit{approximation data for $\mathcal{T}$} when they satisfy the above conditions.
\end{definition}

Suppose $\mathcal{T}$ is such an approximable triangulated category (in particular, we assume by definition that it has all small coproducts). 
Let $\mathcal{T}^{-}_c$ denote the collection of objects $E\in\mathcal{T}$ such that for each non-negative integer $n$ there exists a distinguished triangle
\begin{displaymath}
    P \to E \to F \to G[1]
\end{displaymath}
with $P\in\mathcal{T}^c$ and $F\in\mathcal{T}^{\leq -n}$.
In addition, put $\mathcal{T}^b_c:=\mathcal{T}^{-}_c\cap\mathcal{T}^b$ where $\mathcal{T}^b := (\cup_{n=1}^\infty \mathcal{T}^{\leq n}) \cap (\cup^\infty_{n=1} \mathcal{T}^{\geq -n})$.

Consider the above in our noncommutative set-up, i.e.\ consider a Noether nc.\ scheme $(X,\mathcal{A})$ and take $\mathcal{T}=D_{\operatorname{qc}}(\mathcal{A})$ and $\tau$ the standard $t$-structure. 
This is approximable by \cite{DeDeyn/Lank/ManaliRahul:2024} with data including $\tau$. 
Moreover, we have $\mathcal{T}^b_c = D^b_{\operatorname{coh}}(\mathcal{A})$.
Thus an abstract way of formulating \Cref{question:nc_setup} (and more generally the types of generation questions we considered in this work) would be.

\begin{question}\label{question:abstract}
    Let $\mathcal{T}$ be an approximable triangulated category. 
    \begin{enumerate}
        \item When does $\mathcal{T}^b_c$ admit a strong generator?
        % \item When does $\mathcal{T}$ admit a strong $\oplus$-generator?
        \item When does $\mathcal{T}$ admit a strong $\oplus$-generator from $\mathcal{T}^b_c$?
        \item How are these related? (Again one implication is clear.)
    \end{enumerate}
\end{question}

Of course, given \Cref{ex:gregs} it is necessary to add a satisfactory notion of `Noetherianity' to the hypotheses, as in full generality the last question will otherwise definitely be false.
Moreover, perhaps one can relax approximability to the more general notion of `weak approximability' in the sense of \cite{Neeman:2021b}.

\bibliographystyle{alpha}
\bibliography{mainbib} 

\newcommand{\etalchar}[1]{$^{#1}$}
\begin{thebibliography}{ELFST14}

\bibitem[AG61]{Auslander/Goldman:1961}
Maurice Auslander and Oscar Goldman.
\newblock The {Brauer} group of a commutative ring.
\newblock {\em Trans. Am. Math. Soc.}, 97:367--409, 1961.

\bibitem[Aok21]{Aoki:2021}
Ko~Aoki.
\newblock Quasiexcellence implies strong generation.
\newblock {\em J. Reine Angew. Math.}, 780:133--138, 2021.

\bibitem[ATJLL97]{Alonso/Jeremias/Lipman:1997}
Leovigildo Alonso~Tarr\'{\i}o, Ana Jerem\'{\i}as~L\'{o}pez, and Joseph Lipman.
\newblock Local homology and cohomology on schemes.
\newblock {\em Ann. Sci. \'{E}cole Norm. Sup. (4)}, 30(1):1--39, 1997.

\bibitem[Bal16]{Balmer:2016}
Paul Balmer.
\newblock Separable extensions in tensor-triangular geometry and generalized
  {Quillen} stratification.
\newblock {\em Ann. Sci. {\'E}c. Norm. Sup{\'e}r. (4)}, 49(4):907--925, 2016.

\bibitem[BDG17]{Burban/Drozd/Gavran:2017}
Igor Burban, Yuriy Drozd, and Volodymyr Gavran.
\newblock Minors and resolutions of non-commutative schemes.
\newblock {\em Eur. J. Math.}, 3(2):311--341, 2017.

\bibitem[BDL24]{Bhaduri/Dey/Lank:2023}
Anirban Bhaduri, Souvik Dey, and Pat Lank.
\newblock Preservation for generation along the structure morphism of coherent
  algebras over a scheme.
\newblock \href{https://arxiv.org/abs/2312.02840}{arXiv:2312.02840}, 2024.

\bibitem[BE24]{Brown/Erman:2024}
Michael~K. Brown and Daniel Erman.
\newblock A short proof of the {Hanlon}-{Hicks}-{Lazarev} theorem.
\newblock {\em Forum Math. Sigma}, 12:6, 2024.
\newblock Id/No e56.

\bibitem[Be{\u{\i}}78]{Beilinson:1978}
A.~A. Be{\u{\i}}linson.
\newblock Coherent sheaves on {${\bf P}\sp{n}$}\ and problems in linear
  algebra.
\newblock {\em Funktsional. Anal. i Prilozhen.}, 12(3):68--69, 1978.

\bibitem[BF12]{Ballard/Favero:2012}
Matthew~R. Ballard and David Favero.
\newblock Hochschild dimensions of tilting objects.
\newblock {\em Int. Math. Res. Not.}, 2012(11):2607--2645, 2012.

\bibitem[Bha12]{Bhatt:2012}
Bhargav Bhatt.
\newblock Derived splinters in positive characteristic.
\newblock {\em Compos. Math.}, 148(6):1757--1786, 2012.

\bibitem[BIL{\etalchar{+}}23]{BILMP:2023}
Matthew~R. Ballard, Srikanth~B. Iyengar, Pat Lank, Alapan Mukhopadhyay, and
  Josh Pollitz.
\newblock High frobenius pushforwards generate the bounded derived category.
\newblock \href{https://arxiv.org/abs/2303.18085}{arXiv:2303.18085}, 2023.

\bibitem[Bra14]{Brandenburg:2014}
Martin Brandenburg.
\newblock Tensor categorical foundations of algebraic geometry.
\newblock \href{https://arxiv.org/abs/1410.1716}{arXiv:1410.1716}, 2014.

\bibitem[BS17]{Bhatt/Scholze:2017}
Bhargav Bhatt and Peter Scholze.
\newblock Projectivity of the {Witt} vector affine {Grassmannian}.
\newblock {\em Invent. Math.}, 209(2):329--423, 2017.

\bibitem[BVdB03]{Bondal/VandenBergh:2003}
Alexei Bondal and Michel {V}an~den {B}ergh.
\newblock Generators and representability of functors in commutative and
  noncommutative geometry.
\newblock {\em Mosc. Math. J.}, 3(1):1--36, 258, 2003.

\bibitem[Dan78]{Danilov:1978}
Vladimir~I Danilov.
\newblock The geometry of toric varieties.
\newblock {\em Russian Mathematical Surveys}, 33(2):97, 1978.

\bibitem[DL24a]{Dey/Lank:2024a}
Souvik Dey and Pat Lank.
\newblock Closedness of the singular locus and generation for derived
  categories.
\newblock \href{https://arxiv.org/abs/2403.19564}{arXiv:2403.19564}, 2024.

\bibitem[DL24b]{Dey/Lank:2024}
Souvik Dey and Pat Lank.
\newblock D\'{e}vissage for generation in derived categories.
\newblock \href{https://arxiv.org/abs/2401.13661}{arXiv:2401.13661}, 2024.

\bibitem[DLM24]{DeDeyn/Lank/ManaliRahul:2024}
Timothy~De Deyn, Pat Lank, and Kabeer {Manali Rahul}.
\newblock Approximability and rouquier dimension for noncommutative algebras
  over schemes.
\newblock \href{https://arxiv.org/abs/2408.04561}{arXiv:2408.04561}, 2024.

\bibitem[DLM25]{DeDeyn/Lank/ManaliRahul:2025}
Timothy {De Deyn}, Pat Lank, and Kabeer {Manali Rahul}.
\newblock Descending strong generation in algebraic geometry.
\newblock \href{https://arxiv.org/abs/2502.08629}{arXiv:2502.08629}, 2025.

\bibitem[Dun16]{Duncan:2016}
Alexander Duncan.
\newblock Twisted forms of toric varieties.
\newblock {\em Transform. Groups}, 21(3):763--802, 2016.

\bibitem[EL18]{Elagin/Lunts:2018}
Aleksey~D. Elagin and Valery~A. Lunts.
\newblock Regular subcategories in bounded derived categories of affine
  schemes.
\newblock {\em Sb. Math.}, 209(12):1756--1782, 2018.

\bibitem[ELFST14]{Elizondo/Lima-Filho/Sottile/Teitler:2014}
E.~Javier Elizondo, Paulo Lima-Filho, Frank Sottile, and Zach Teitler.
\newblock Arithmetic toric varieties.
\newblock {\em Math. Nachr.}, 287(2-3):216--241, 2014.

\bibitem[ELS20]{Elagin/Lunts/Schnurer:2020}
Alexey Elagin, Valery~A. Lunts, and Olaf~M. Schn\"urer.
\newblock Smoothness of derived categories of algebras.
\newblock {\em Mosc. Math. J.}, 20(2):277--309, 2020.

\bibitem[FH23]{Favero/Huang:2023}
David Favero and Jesse Huang.
\newblock Rouquier dimension is {Krull} dimension for normal toric varieties.
\newblock {\em Eur. J. Math.}, 9(4):13, 2023.
\newblock Id/No 91.

\bibitem[Gro61]{EGAIII1:1961}
A.~Grothendieck.
\newblock \'el\'ements de g\'eom\'etrie alg\'ebrique. {III}. \'etude
  cohomologique des faisceaux coh\'erents. {I}.
\newblock {\em Inst. Hautes \'Etudes Sci. Publ. Math.}, (11):167, 1961.

\bibitem[Hal16]{Hall:2016}
Jack Hall.
\newblock The {Balmer} spectrum of a tame stack.
\newblock {\em Ann. \(K\)-Theory}, 1(3):259--274, 2016.

\bibitem[HHL23]{Hanlon/Hicks/Lazarev:2023}
Andrew Hanlon, Jeff Hicks, and Oleg Lazarev.
\newblock Resolutions of toric subvarieties by line bundles and applications.
\newblock \href{https://arxiv.org/abs/2303.03763}{arXiv:2303.03763}, 2023.

\bibitem[HR17]{Hall/Rydh:2017}
Jack Hall and David Rydh.
\newblock Perfect complexes on algebraic stacks.
\newblock {\em Compos. Math.}, 153(11):2318--2367, 2017.

\bibitem[IT19]{Iyengar/Takahashi:2019}
Srikanth~B. Iyengar and Ryo Takahashi.
\newblock Openness of the regular locus and generators for module categories.
\newblock {\em Acta Math. Vietnam.}, 44(1):207--212, 2019.

\bibitem[Kov00]{Kovacs:2000}
S{\'a}ndor~J. Kov{\'a}cs.
\newblock A characterization of rational singularities.
\newblock {\em Duke Math. J.}, 102(2):187--191, 2000.

\bibitem[Kra05]{Krause:2005}
Henning Krause.
\newblock The stable derived category of a {N}oetherian scheme.
\newblock {\em Compos. Math.}, 141(5):1128--1162, 2005.

\bibitem[Kun76]{Kunz:1976}
Ernst Kunz.
\newblock On {Noetherian} rings of characteristic p.
\newblock {\em Am. J. Math.}, 98:999--1013, 1976.

\bibitem[Kuz06]{Kuznetsov:2006}
Alexander~G. Kuznetsov.
\newblock Hyperplane sections and derived categories.
\newblock {\em Izv. Ross. Akad. Nauk Ser. Mat.}, 70(3):23--128, 2006.

\bibitem[Lan24]{Lank:2024}
Pat Lank.
\newblock Descent conditions for generation in derived categories.
\newblock {\em J. Pure Appl. Algebra}, 228(9):Paper No. 107671, 19, 2024.

\bibitem[Let21]{Letz:2021}
Janina~C. Letz.
\newblock Local to global principles for generation time over commutative
  {Noetherian} rings.
\newblock {\em Homology Homotopy Appl.}, 23(2):165--182, 2021.

\bibitem[LMV25]{Lank/McDonald/Venkatesh:2025}
Pat Lank, Peter McDonald, and Sridhar Venkatesh.
\newblock Derived characterizations for rational pairs \`{a} la schwede-takagi
  and koll\'{a}r-kov\'{a}cs.
\newblock \href{https://arxiv.org/abs/2501.02783}{arXiv:2501.02783}, 2025.

\bibitem[LO24]{Lank/Olander:2024}
Pat Lank and Noah Olander.
\newblock Approximation by perfect complexes detects rouquier dimension.
\newblock \href{https://arxiv.org/abs/2401.10146}{arXiv:2401.10146}, 2024.

\bibitem[LV24]{Lank/Venkatesh:2024}
Pat Lank and Sridhar Venkatesh.
\newblock Triangulated characterizations of singularities.
\newblock \href{https://arxiv.org/abs/2405.04389}{arXiv:2405.04389}, 2024.

\bibitem[Mat16]{Mathew:2016}
Akhil Mathew.
\newblock The {Galois} group of a stable homotopy theory.
\newblock {\em Adv. Math.}, 291:403--541, 2016.

\bibitem[Mil80]{Milne:1980}
James~S. Milne.
\newblock {\em \'Etale cohomology}, volume No. 33 of {\em Princeton
  Mathematical Series}.
\newblock Princeton University Press, Princeton, NJ, 1980.

\bibitem[MP97]{Merkurjev/Panin:1997}
A.~S. Merkurjev and I.~A. Panin.
\newblock {{\(K\)}}-theory of algebraic tori and toric varieties.
\newblock {\em \(K\)-Theory}, 12(2):101--143, 1997.

\bibitem[Nee92]{Neeman:1992}
Amnon Neeman.
\newblock The chromatic tower for {{\(D(R)\)}}. {With} an appendix by {Marcel}
  {B{\"o}kstedt}.
\newblock {\em Topology}, 31(3):519--532, 1992.

\bibitem[Nee96]{Neeman:1996}
Amnon Neeman.
\newblock The {Grothendieck} duality theorem via {Bousfield}'s techniques and
  {Brown} representability.
\newblock {\em J. Am. Math. Soc.}, 9(1):205--236, 1996.

\bibitem[Nee18]{Neeman:2018b}
Amnon Neeman.
\newblock The category {$[\mathcal{T}^c]^{\mathrm{op}}$} as functors on
  {$\mathcal{T}^b_c$}.
\newblock \href{https://arxiv.org/abs/1806.05777}{arXiv:1806.05777}, 2018.

\bibitem[Nee21a]{Neeman:2021a}
Amnon Neeman.
\newblock Approximable triangulated categories.
\newblock In {\em Representations of algebras, geometry and physics, Maurice
  Auslander distinguished lectures and international conference, Woods Hole
  Oceanographic Institute, Woods Hole, MA, USA, April 25--30, 2018}, pages
  111--155. Providence, RI: American Mathematical Society (AMS), 2021.

\bibitem[Nee21b]{Neeman:2021b}
Amnon Neeman.
\newblock Strong generators in {{\(\mathbf{D}^{\mathrm{perf}}(X)\)}} and
  {{\(\mathbf{D}^b_{\mathrm{coh}}(X)\)}}.
\newblock {\em Ann. Math. (2)}, 193(3):689--732, 2021.

\bibitem[Ola23]{Olander:2023}
Noah Olander.
\newblock Ample line bundles and generation time.
\newblock {\em J. Reine Angew. Math.}, 800:299--304, 2023.

\bibitem[Orl09]{Orlov:2009}
Dmitri Orlov.
\newblock Remarks on generators and dimensions of triangulated categories.
\newblock {\em Mosc. Math. J.}, 9(1):143--149, 2009.

\bibitem[Rou08]{Rouquier:2008}
Rapha{\"e}l Rouquier.
\newblock Dimensions of triangulated categories.
\newblock {\em J. \(K\)-Theory}, 1(2):193--256, 2008.

\bibitem[Sos14]{Sosna:2014}
Pawel Sosna.
\newblock Scalar extensions of triangulated categories.
\newblock {\em Appl. Categ. Struct.}, 22(1):211--227, 2014.

\bibitem[Spa88]{Spaltenstein:1988}
Nicolas Spaltenstein.
\newblock Resolutions of unbounded complexes.
\newblock {\em Compositio Math.}, 65(2):121--154, 1988.

\bibitem[{Sta}25]{stacks-project}
The {Stacks project authors}.
\newblock The stacks project.
\newblock \url{https://stacks.math.columbia.edu}, 2025.

\bibitem[Ste13]{Stevenson:2013}
Greg Stevenson.
\newblock Support theory via actions of tensor triangulated categories.
\newblock {\em J. Reine Angew. Math.}, 681:219--254, 2013.

\bibitem[Ste14]{Stevenson:2014}
Greg Stevenson.
\newblock Subcategories of singularity categories via tensor actions.
\newblock {\em Compos. Math.}, 150(2):229--272, 2014.

\bibitem[Ste25]{Stevenson:2025}
Greg Stevenson.
\newblock Rouquier dimension versus global dimension.
\newblock {\em J. Pure Appl. Algebra}, 229(1):Paper No. 107827, 4, 2025.

\bibitem[To{\"{e}}12]{Toen:2012}
Bertrand To{\"{e}}n.
\newblock Derived {A}zumaya algebras and generators for twisted derived
  categories.
\newblock {\em Invent. Math.}, 189(3):581--652, 2012.

\bibitem[VdB04]{VandenBergh:2004}
Michel Van~den Bergh.
\newblock Non-commutative crepant resolutions.
\newblock In {\em The legacy of {N}iels {H}enrik {A}bel}, pages 749--770.
  Springer, Berlin, 2004.

\bibitem[VK85]{Voskresenskij/Klyachko:1985}
V.~E. Voskresenskij and A.~A. Klyachko.
\newblock Toroidal {Fano} varieties and root systems.
\newblock {\em Math. USSR, Izv.}, 24:221--244, 1985.

\bibitem[Voe96]{Voevodsky:1996}
Vladimir Voevodsky.
\newblock Homology of schemes.
\newblock {\em Selecta Math. (N.S.)}, 2(1):111--153, 1996.

\end{thebibliography}
\end{document}